\documentclass[11pt]{article}
\usepackage{amsfonts}
\usepackage{xcolor}
\usepackage{amsmath,amssymb,comment}
\newtheorem{theo}{Theorem}[section]
\newtheorem{lem}[theo]{Lemma}
\newtheorem{cor}[theo]{Corollary}
\newtheorem{prop}[theo]{Proposition}

\newtheorem{defi}[theo]{Definition}

\newcommand{\mysection}[1]{\section{#1} \setcounter{equation}{0}}
\newcommand{\proof}{{\sc Proof.} \quad}
\newcommand{\proofc}{{\sc Proof} \ }
\newcommand{\be}{\begin{equation} \label}
\newcommand{\ee}{\end{equation}}
\newcommand{\bea}{\begin{eqnarray}\label}
\newcommand{\eea}{\end{eqnarray}}
\newcommand{\bas}{\begin{eqnarray*}}
\newcommand{\eas}{\end{eqnarray*}}
\newcommand{\bit}{\begin{itemize}}
\newcommand{\eit}{\end{itemize}}
\newcommand{\qed}{\hfill$\Box$ \vskip.2cm}
\newcommand{\nn}{\nonumber}

\newcommand{\R}{\mathbb{R}}
\newcommand{\N}{\mathbb{N}}
\newcommand{\pO}{\partial\Omega}
\newcommand{\eps}{\varepsilon}

\newcommand{\supp}{{\rm supp} \, }
\newcommand{\sgn}{{\rm sgn} \, }

\newcommand{\io}{\int_\Omega}
\newcommand{\na}{\nabla}
\newcommand{\Del}{\Delta}
\newcommand{\del}{\delta}
\newcommand{\al}{\alpha}
\newcommand{\lam}{\lambda}

\newcommand{\sig}{\sigma}
\newcommand{\vt}{\vartheta}

\newcommand{\pa}{\partial}
\newcommand{\bom}{\overline{\Omega}}
\newcommand{\Om}{\Omega}

\newcommand{\ov}{\overline}
\newcommand{\un}{\underline}
\newcommand{\wh}{\widehat}

\newcommand{\hs}{\hspace*}
\newcommand{\sm}{\setminus}

\newcommand{\vp}{\varphi}
\newcommand{\lbal}{\left\{ \begin{array}{l}}
\newcommand{\lball}{\left\{ \begin{array}{ll}}
\newcommand{\ear}{\end{array} \right.}

\newcommand{\abs}{\\[5pt]}
\newcommand{\Abs}{\\[5mm]}

\newcommand{\adb}{\allowdisplaybreaks}
\newcommand{\oy}{\ov{y}}

\newcommand{\ou}{\ov{u}}
\newcommand{\uu}{\un{u}}

\newcommand{\CC}{{\mathcal{C}}}
\begin{document}
\adb
\title{Blow-down and immediate onset of positivity\\
in the Neumann problem for (super-)fast diffusion equations}
\author
{
Youshan Tao\footnote{taoys@sjtu.edu.cn}\\
{\small School of Mathematical Sciences, CMA-Shanghai, Shanghai Jiao Tong University,}\\
{\small Shanghai 200240, P.R.~China}
\and
Michael Winkler\footnote{michael.winkler@math.uni-paderborn.de}\\
{\small Universit\"at Paderborn, Institut f\"ur Mathematik,}\\
{\small 33098 Paderborn, Germany} }
\date{}
\maketitle
\begin{abstract}
\noindent
  For the degenerate parabolic equation $u_t=\nabla\cdot (u^{m-1}\nabla u)$,
  the homogeneous Neumann problem in smoothly bounded domains $\Omega\subset\R^n$
  is considered in the (super-fast) diffusion range when $m\in (-\infty,1)$.
  Upper and lower estimates for solutions are inferred by means of two Moser-type iterations,
  and optimality properties of these bounds are derived using comparison arguments.
  On the basis of the accordingly obtained information on regularization into spaces of bounded and uniformly positive functions,
  a result on global solvability in an associated initial-boundary value problem is achieved for all nonnegative and nontrivial
  initial data belonging to $L^1(\Om) \cap \bigcup_{s>n(1-m)/2} L^s(\Om)$.\abs
\end{abstract}
%
%
%
%
%
%
%
%
%
\section{Introduction}\label{intro}
This manuscript is concerned with nonnegative solutions of the nonlinear parabolic equation
\be{00}
	u_t = \na \cdot (u^{m-1}\na u)
\ee
in the so-called fast diffusion range where $m\in (-\infty,1)$.
Together with close relatives, this equation arises, for instance, 
in plasma and semiconductor physics (\cite{3}, \cite{4}, \cite{14}, \cite{18}), in theories for diffusive limits
of kinetic models (\cite{17}), and in differential geometry (\cite{12}); further application contexts are
described in \cite{V1} and \cite{DK}.
Accordingly, an extensive literature has been devoted to the analysis of core properties resulting from
the two diffusion degeneracies arising both near points where $u$ is large,
and in regions where $u$ is small. (\cite{V1}, \cite{V2}, \cite{DK}).\abs
In the framework of the Cauchy problem for (\ref{00}) in $\R^n$ with $n\ge 1$, for instance, classical results have identified
the exponent $m_c:=\frac{(n-2)_+}{n}$ as critical: In the supercritical range where $m\in (m_c,1)$,
actually extending to all of $(m_c,\infty)$ when theories of linear heat and porous medium equations are included,
nonnegative integrable initial
data $u_0$ generate unique solutions which conserve mass and immediately enter $L^\infty(\R^n)$;
specifically, if $u_0\in L^s(\R^n)\cap L^1(\R^n)$ with some $s\ge 1$, then there exists $C=C(u_0)>0$ such that
\be{01}
	\|u(\cdot,t)\|_{L^\infty(\R^n)}
	\le C t^{-\theta(m,s)},
	\quad
	\theta(m,s):=\frac{n}{2s+n(m-1)},
	\qquad \mbox{for all } t>0;
\ee
in fact, even for initial data which merely belong to $L^1_{loc}(\R^n)$, some global solutions exist which are bounded for $t>0$
and satisfy an inequality of the form in (\ref{01}) at least for $t\in (0,1)$
(cf.~\cite{HerreroPierre1985}, \cite{BV_ADM2010}, \cite{V1} and \cite{V2} as well as the references therein, for instance,
and \cite{ERV88} for
an additional result addressing global solutions when $n=1$ and $m>-1$).
When $m\in (0,m_c]$, the estimate in (\ref{01}) continues to hold but is then restricted to initial data belonging
to $L^s(\R^n)$ with $s>\frac{n(1-m)}{2}$ (\cite{V2}).\abs
In this subcritical range $m\in (0,m_c]$, however, including the value $m=0$ if $n=2$,
the apparently most striking feature of (\ref{00}) is linked to its degeneracy near $u=0$:
Throughout large classes of initial data, namely, solutions gradually lose mass and vanish in all of $\R^n$ at some finite time;
for a small selection of findings on the rich variety of possibilities in asymptotic behavior near the extinction time
of such solutions, we may refer to
\cite{del pino_iumj2001}, \cite{BBDGV}, \cite{BGV10}, \cite{BGV12}, \cite{jin_xiong}, \cite{akagi_cmp2016},
\cite{FVWY} and \cite{FKW}), for instance.
In the super-fast diffusion regime where $m\le 0$ if $n\ge 3$, $m<0$ if $n=2$, and $m\le -1$ if $n=1$,
finally, mass escape at spatial inifinity is strong enough to even lead to immediate
extinction: For initial data in $L^1(\R^n)$, no solution exists on any time interval (\cite{daskalopoulos_delpino}, \cite{ERV88}).\abs
{\bf The Neumann problem for (\ref{00}): Global solvability, upper and lower estimates for arbitrary $m\in (-\infty,1)$.} \quad
For homogeneous Dirichlet problems associated with (\ref{00}) in bounded domains,
results similar to the above hold, e.g.~with regard to questions concerning upper estimates of the form in (\ref{01}), as well
as on nonexistence when $m\le 0$, with a slight difference consisting in the fact that also when $m>0$, any solution
becomes extinct within finite time (see \cite{BF} for an overview, as well as
\cite{jin_xiong2026}, \cite{jin_xiong2023} and \cite{jin_xiong2025} for some recent findings on regularity properties and
subtle aspects of extinction behavior).\abs
It is to be anticipated, however, that the situation should significantly change when imposing homogeneous boundary conditions
of Neumann type prevents mass loss through boundaries:
Then, namely, already the presence of constant solutions at arbitrary size levels
strongly indicates that bounded and uniformly positive initial data should evolve into globally regular solutions,
and viewing mass conservation together with nonlinear enhancement of diffusion at small densities may lead to the prediction
that also merely nonnegative nontrivial initial data lead to solutions which immediately become positive.
A change concerning technical aspects can be expected with regard to upper estimates of the form in (\ref{00}):
While both in Cauchy and in homogeneous Dirichlet problems for (\ref{00}) a reasoning can be built on
symmetrization and concentration comparison with self-similar solutions of Barenblatt type (\cite{V2}),
no such reduction seems available in Neumann problems posed in widely general smoothly bounded domains.\abs
For sufficiently regular initial data, corresponding difficulties have been overcome in \cite{IPS},
actually in the context of the Neumann problem in a smoothly bounded domain $\Om\subset\R^n$ for a generalization of 
(\ref{00}) with $m<0$
that additionally contains some prescribed space-dependent weight functions in the diffusion rate,
by means of an approach interpreting (\ref{00}) as a gradient flow with respect to a Wasserstein distance:
For initial data $u_0 \in L^{3-m}(\Om)$ which are essentially positive in the sense that the associated energy $\io u_0^m$
is finite, some unique global weak solutions are constructed (\cite[Theorem 1.2]{IPS}).
Under the stronger assumption that, in addition to the above,
\be{02}
	\io u_0^s < \infty
	\mbox{ for some } s>\frac{n(1-m)}{2}
	\qquad \mbox{and} \qquad
	\io u_0^{-p} < \infty
	\mbox{ for some } p>(1-m)\cdot\max\Big\{1,\frac{n}{2}\Big\},
\ee
these solutions are shown to have the property that both $u$ and $\frac{1}{u}$ are bounded in $\Om\times (\tau,\infty)$
for each $\tau>0$ (\cite[Theorem 1.3]{IPS}; cf.~also 
\cite{DS} for supplementary information on large time behavior).
It seems widely open, however, how the fast diffusion mechanism in (\ref{00}), especially in the super-fast diffusion range $m\le 0$,
acts on less regular initial data, and particularly on initial data with compact support.\abs
The first purpose of the present manuscript consists in developing a basic theory of global solvability for such problems
also in the presence of merely nonnegative initial data.
In close connection to this, a second objective is
to provide quantitative information not only on upper bounds for solutions, but especially also on the immediate onset
of uniform positivity, as forming a feature of distinctly novel character when compared with corresponding properties
of Cauchy and zero Dirichlet problems.
Thus subsequently concerned with
\be{0}
	\lball
	u_t = \na\cdot (u^{m-1}\na u),
	\qquad & x\in\Om, \ t>0, \\[1mm]
	\frac{\pa u}{\pa\nu}=0,
	\qquad & x\in\pO, \ t>0, \\[1mm]
	u(x,0)=u_0(x),
	\qquad & x\in\Om,
	\ear
\ee
in a smoothly bounded domain $\Om\subset\R^n$ with $n\ge 1$ and with $m\in (-\infty,1)$,
we will address these issues by firstly operating in the framework of the following fairly intuitive concept of solvability:
\begin{defi}\label{ds}
  Let $n\ge 1$ and $\Om\subset\R^n$ be a bounded domain with smooth boundary, let $m\in (-\infty,1)$, and suppose
  that $u_0\in L^1(\Om)$ be nonnegative.
  Then by a {\em locally smooth $L^1$-continuous solution of (\ref{0})} we mean a function
  $u\in C^0([0,\infty);L^1(\Om)) \cap C^{2,1}(\bom\times (0,\infty))$ such that $u>0$ in $\bom\times (0,\infty)$,
  that in the classical sense we have $u_t=\na\cdot (u^{m-1}\na u)$ in $\Om\times (0,\infty)$ and
  $\frac{\pa u}{\pa\nu}=0$ on $\pO\times (0,\infty)$, and that $u(\cdot,t)\to u_0$ in $L^1(\Om)$ as $t\searrow 0$.
\end{defi}
Indeed, this concept turns out to be meaningful in the sense that it features existence and uniqueness
for arbitrary nonnegative and nontrivial initial data which satisfy
an integrability condition consistent with the requirement that the exponent $\theta(m,s)$ in (\ref{01}) be well-defined and positive:
\begin{prop}\label{prop66}
  Let $n\ge 1$ and $\Om\subset\R^n$ be a bounded domain with smooth boundary, and let $m\in (-\infty,1)$ and $s\ge 1$ be such that
  \be{66.1}
	s>\frac{n(1-m)}{2}.
  \ee
  Then for any nonnegative $u_0\in L^s(\Om)$ with $u_0\not\equiv 0$, the problem (\ref{0}) admits a uniquely determined
  locally smooth $L^1$-continuous solution $u$ in the sense of Definition \ref{ds}.\\
  This solution can be approximated by classical solutions of (\ref{0}) with smooth initial data in the following sense:
  Whenever $(u_{0j})_{j\in\N} \subset C^\infty(\bom)$ is such that $u_{0j}>0$ in $\bom$ for all $j\in\N$ and $u_{0j} \to u_0$
  in $L^s(\Om)$ as $j\to\infty$, the corresponding solutions $u_j$ of (\ref{0}) with initial data $u_{0j}$ satisfy
  $u_j \to u$ in $C^0_{loc}([0,\infty);L^1(\Om))$ and in $C^{2,1}_{loc}(\bom\times (0,\infty))$ as $j\to\infty$.
\end{prop}
In Section \ref{sect4}, this result will be derived from two results of Moser-type iterative $L^p$ arguments involving
time-dependent weights.
The first of these, to be verified in Section \ref{sect2}, will as a by-product
imply the following quantitative blow-down estimate for solutions
in dependence on regularity of the initial data, measured in $L^s$ spaces with $s\ge 1$ satisfying (\ref{66.1}).
\begin{theo}\label{theo67}
  Let $n\ge 1$ and $\Om\subset\R^n$ be a bounded domain with smooth boundary, and let $m\in (-\infty,1)$, $s\in [1,\infty)$
  and $p\in (1,\infty]$ be such that
  \be{67.1}
	p>s>\frac{n(1-m)}{2}
	\qquad \mbox{and} \qquad
	p>1-m.
  \ee
  Then there exists $\Gamma^\star=\Gamma^\star(m,s,p)>0$ with the property that whenever $u_0\in L^s(\Om)$ is nonnegative with
  $u_0\not\equiv 0$, the locally smooth $L^1$-continuous solution of (\ref{0}) satisfies
  \be{67.2}
	\|u(\cdot,t)\|_{L^p(\Om)}
	\le \Gamma^\star \|u_0\|_{L^s(\Om)}^\frac{2s+\frac{n(m-1)s}{p}}{2s+n(m-1)} \cdot t^{-\frac{n(1-\frac{s}{p})}{2s+n(m-1)}}
		+ \Gamma^\star \|u_0\|_{L^s(\Om)}
	\qquad \mbox{for all } t>0,
  \ee
  where we have set $\frac{\xi}{p}:=0$ for $\xi\in\R$ if $p=\infty$
\end{theo}
We underline that in the porous medium type counterpart of the Neumann problem (\ref{0}) in which $m>1$,
in the special case when $p=\infty$ an estimate precisely of the form then appearing in (\ref{67.2})
was found in \cite{grillo_jde2013} by means of a Moser recursion slightly different from the one to be performed below.
This technical deviation undertaken here is motivated by the ambition to set up an approach which upon some convenient modification
can be adapted so as to address the actually more delicate problem of establishing pointwise lower bounds for solutions.
Based on a variant of a classical Poincar\'e inequality explicitly spotlighting dependencies on sizes
of zero sets (cf.~Lemma \ref{lem11}), in Section \ref{sect3}
the following quantitative information on the immediate onset of spatially uniform
positivity will arise as a second by-product from our path toward Proposition \ref{prop66}.
\begin{theo}\label{theo68}
  If $n\ge 1$ and $\Om\subset\R^n$ is a bounded domain with smooth boundary, and if $m\in (-\infty,1)$, $q\in (0,1]$
  and $r\in [1,\infty)$ are such that $q<r$,
  then one can choose $\Gamma_\star=\Gamma_\star(m,q,r)>0$ in such a way that whenever
  $u_0\in L^1(\Om) \cap \bigcup_{s>n(1-m)/2} L^s(\Om)$ is nonnegative with $u_0\not\equiv 0$, the locally smooth
  $L^1$-continuous solution $u$ of (\ref{0}) has the property that
  \be{68.2}
	\Big\| \frac{1}{u(\cdot,t)} \Big\|_{L^\infty(\Om)}
	\le \Gamma_\star \cdot
	\zeta^\star
	\cdot t^{-\frac{1}{1-m}}
	+ \Gamma_\star \cdot
	\zeta^\star
	\cdot \frac{1}{\|u_0\|_{L^q(\Om)}}
	\quad \mbox{for all } t>0
  \ee
  with
  \be{zs}
	\zeta^\star:=\zeta_{\frac{qr}{r-q}} \Big( \mbox{$\frac{\|u_0\|_{L^q(\Om)}}{\|u_0\|_{L^r(\Om)}}$} \Big),
  \ee
  where for $\vt>0$ and $\xi>0$ we have set
  \be{zeta}
	\zeta_\vt(\xi) := \lball
	\xi^{-\frac{(n-2)\vt}{n}}
	\qquad & \mbox{if } n\ge 3, \\[1mm]
	1+ \ln_+ \frac{1}{\xi}
	\qquad & \mbox{if } n=2, \\[1mm]
	1
	\qquad & \mbox{if } n=1,
	\ear
  \ee
  with $\ln_+ \sig:=\max \{\ln \sig, 0\}$ for $\sig>0$.
\end{theo}
Independently of Theorems \ref{theo67} and \ref{theo68}, a simple observation will establish a link between the solution
framework considered above and the following more standard notion which appears to be a straightforward extension of the
very traditional setting from, e.g., \cite{ACP} to the present situation.
\begin{defi}\label{dw}
  Let $n\ge 1$ and $\Om\subset\R^n$ be a bounded domain with smooth boundary, let $m\in (-\infty,1)$, let
  $u_0\in L^1(\Om)$, and suppose that
  $u\in C^0([0,\infty);L^1(\Om))$ is nonnegative and such that $\Psi(u)\in L^\infty_{loc}(\bom\times (0,\infty))$, where
  \be{Psi}
	\Psi(\xi):=
	\lball
	\frac{1}{m} \xi^m, \quad \xi\ge 0,
	\qquad & \mbox{if } m>0, \\[2mm]
		\lball
		\ln \xi,
		\quad & \xi>0, \\[1mm]
		-\infty
		\quad & \xi=0,
		\ear
	\qquad & \mbox{if } m=0, \\[5mm]
		\lball
		\frac{1}{m} \xi^m
		\quad & \xi>0, \\[1mm]
		-\infty
		\quad & \xi=0,
		\ear
	\qquad & \mbox{if } m<0.
	\ear
  \ee
  Then $u$ will be called a {\em weak solution of (\ref{0})} if for arbitrary $T>0$ and $t_0\in (0,T)$, and for
  each $\vp\in C^{2,1}(\bom\times [t_0,T])$ fulfilling $\frac{\pa\vp}{\pa\nu}=0$ on $\pO\times [t_0,T]$, the identity
  \be{0w}
	- \int_{t_0}^T \io u\vp_t + \io u(\cdot,T)\vp(\cdot,T) - \io u(\cdot,t_0) \vp(\cdot,t_0)
	= \int_{t_0}^T \io \Psi(u) \Del\vp
  \ee
  is valid.
\end{defi}
Indeed:
\begin{cor}\label{cor57}
  Let $n\ge 1$ and $\Om\subset\R^n$ be a bounded domain with smooth boundary, let $m\in (-\infty,1)$,
  and suppose that $u_0\in L^1(\Om) \cap \bigcup_{s>n(1-m)/2} L^s(\Om)$ is nonnegative with $u_0\not\equiv 0$.
  Then (\ref{0}) possesses a unique weak solution $u$ in the sense of Definition \ref{dw}.
  This solution coincides with the locally smooth $L^1$-continuous solution constructed above, and particularly
  satisfies $u>0$ in $\bom\times [0,\infty)$ as well as $u\in C^{2,1}(\bom\times (0,\infty))$.
\end{cor}
{\bf Remark.} \quad
i) \ For nontrivial nonnegative initial data $u_0\in L^1(\Om) \cap \bigcup_{s>n(1-m)/2} L^s(\Om)$,
the existence results in Proposition \ref{prop66} and Corollary \ref{cor57} extend those in \cite{IPS} by admitting
arbitrary behavior of $u_0$ near $u_0=0$; in particular, compactly supported initial data are covered here.\abs
ii) \ Theorems \ref{theo67} and \ref{theo68} go beyond the two-sided estimates in \cite[Theorem 1.3]{IPS}
not only by relying on assumptions on the initial data that are significantly weaker than those in (\ref{02}),
but also by providing quantitative information on the rates at which the respective instantaneous regularization effects occur.\Abs
Let us finally provide some evidence indicating that the statements from Proposition \ref{prop66} and Theorem \ref{theo67}
cannot substantially be improved.
Firstly, in the range $m<\frac{n-2}{n}$ in which there exists $s\ge 1$ satisfying $s<\frac{n(1-m)}{2}$, one can find
initial data which belong to $L^s(\Om)$ for any such $s$, but for which comparison from below with persistently singular
subsolutions will show in Section \ref{sect5} that (\ref{0}) does not possess any in the sense of Definition \ref{ds}:
\begin{prop}\label{prop56}
  Let $n\ge 1$ and $\Om\subset\R^n$ be a bounded domain with smooth boundary, and let $m\in (-\infty,\frac{n-2}{n})$.
  Then for any $x_0\in\Om$ there exists a nonnegative $u_0\in \bigcap_{s\in [1,n(1-m)/2)} L^s(\Om) \cap C^\infty(\bom\sm\{0\})$
  such that (\ref{0}) does not admit any locally smooth $L^1$-continuous solution $u$.
  This conclusion actually holds for any $u_0\in L^1(\Om)$ satisfying
  \be{56.1}
	{\rm{ess}} \inf_{\!\!\! x\in\Om} \Big\{ |x-x_0|^\frac{2}{1-m} u_0(x) \Big\} >0.
  \ee
\end{prop}
Another comparison argument based on suitably constructed smooth supersolutions will reveal that also the time dependence
in the lower estimate found in Theorem \ref{theo68} is essentially optimal:
\begin{prop}\label{prop55}
  Let $n\ge 1$ and $\Om\subset\R^n$, let $m\in (-\infty,1)$ and let $\eps\in (0,1)$.
  Then there exist $C>0$ and $u_0\in C^\infty(\bom)$ such that $u_0\ge 0$ and $u_0\not\equiv 0$ in $\Om$, and that the corresponding
  locally smooth $L^1$-continuous solution $u$ of (\ref{0}) satisfies
  \be{55.1}
	\inf_{x\in\Om} u(x,t) \le C t^\frac{1-\eps}{1-m}
	\qquad \mbox{for all } t>0.
  \ee
\end{prop}
\mysection{Upper bounds. Regularization into $L^p$ and $L^\infty$}\label{sect2}
This section will not only undertake a first step toward the existence statement in Proposition \ref{prop66}, but at the same
time also prepare the most substantial part of the quantitative claim in Theorem \ref{theo67}.
A fundament for our pointwise upper estimates, and thereby also a basis for the Moser type argument performed in Lemma \ref{lem2},
is obatined by a fairly classical combination of a standard testing procedure with a Gagliardo-Nirenberg interpolation
relying on the nonincrease of $0\le t\mapsto \io u^s(\cdot,t)$ along smooth trajectories for each $s\ge 1$:
\begin{lem}\label{lem1}
  Let $n\ge 1, m\in (-\infty,1), s\ge 1$ and $p\ge 1$ be such that
  \be{1.1}
	p>s>\frac{n(1-m)}{2}
	\qquad \mbox{and} \qquad
	p>1-m.
  \ee
  Then there exists $\Gamma_1=\Gamma_1(m,p,s)>0$ with the property that whenever $u_0\in C^0(\bom)$ is positive in $\bom$
  and $u\in C^0(\bom\times [0,\infty)) \cap C^{2,1}(\bom\times (0,\infty))$ is such that $u>0$ in $\bom\times (0,\infty)$
  and that (\ref{0}) is solved in the classical sense, we have
  \be{1.2}
	\|u(\cdot,t)\|_{L^p(\Om)}
	\le \Gamma_1 \|u_0\|_{L^s(\Om)}^\frac{2s+\frac{n(m-1)s}{p}}{2s+n(m-1)} \cdot t^{-\frac{n(1-\frac{s}{p})}{2s+n(m-1)}}
		+ \Gamma_1 \|u_0\|_{L^s(\Om)}
	\qquad \mbox{for all } t>0.
  \ee
\end{lem}
\proof
  According to (\ref{0}), we have
  \be{1.3}
	\frac{d}{dt} \io u^p = -p(p-1) \io u^{p+m-3} |\na u|^2
	= -c_1 \io\big|\na u^\frac{p+m-1}{2}\big|^2
	\qquad \mbox{for all } t>0,
  \ee
  where $c_1\equiv c_1(m,p):=\frac{4p(p-1)}{(p+m-1)^2}$ is well-defined and positive, because $p>s\ge 1$ and $p+m-1>0$ by
  (\ref{1.1}).
  Since the inequality $s\ge 1$ furthermore ensures that $\frac{d}{ds} \io u^s \le 0$ for all $t>0$ and hence
  \bas
	\|u(\cdot,t)\|_{L^s(\Om)} \le K:=\|u_0\|_{L^s(\Om)}
	\qquad \mbox{for all } t>0,
  \eas
  we may here rely on the Gagliardo-Nirenberg inequality, which namely provides $c_2=c_2(m,p,s)>0$ such that
  with
  \be{1.33}
	\iota:=\frac{n(p-s)(p+m-1)}{p\cdot [(2-n)s+n(p+m-1)]} \in (0,1)
  \ee
  we have
  \bas
	\|\vp\|_{L^\frac{2p}{p+m-1}(\Om)}^\frac{2p}{p+m-1}
	\le c_2 \|\na\vp\|_{L^2(\Om)}^\frac{2p\iota}{p+m-1} \|\vp\|_{L^\frac{2s}{p+m-1}(\Om)}^\frac{2p(1-\iota)}{p+m-1}
		+ c_2\|\vp\|_{L^\frac{2s}{p+m-1}(\Om)}^\frac{2p}{p+m-1}
	\qquad \mbox{for all } \vp\in W^{1,2}(\Om),
  \eas
  to see that
  \bas
	\io u^p
	&=& \big\| u^\frac{p+m-1}{2} \big\|_{L^\frac{2p}{p+m-1}(\Om)}^\frac{2p}{p+m-1}  \\
	&\le& c_2 \big\| \na u^\frac{p+m-1}{2} \big\|_{L^2(\Om)}^\frac{2p\iota}{p+m-1}
		\big\| u^\frac{p+m-1}{2}\big\|_{L^\frac{2s}{p+m-1}(\Om)}^\frac{2p(1-\iota)}{p+m-1}
	+ c_2 \big\| u^\frac{p+m-1}{2} \big\|_{L^\frac{2s}{p+m-1}(\Om)}^\frac{2p}{p+m-1} \\
	&=& c_2 \big\| \na u^\frac{p+m-1}{2} \big\|_{L^2(\Om)}^\frac{2p\iota}{p+m-1}
		\|u\|_{L^s(\Om)}^{p(1-\iota)}
	+ c_2 \|u\|_{L^s(\Om)}^p \\
	&\le& c_2 K^{p(1-\iota)} \big\| \na u^\frac{p+m-1}{2} \big\|_{L^2(\Om)}^\frac{2p\iota}{p+m-1} + c_2 K^p
	\qquad \mbox{for all } t>0.
  \eas
  Therefore,
  \bas
	\bigg\{ \io u^p \bigg\}^\frac{p+m-1}{p\iota}
	\le c_3 K^\frac{(p+m-1)(1-\iota)}{\iota} \big\| \na u^\frac{p+m-1}{2} \big\|_{L^2(\Om)}^2
	+ c_3 K^\frac{p+m-1}{\iota}
	\qquad \mbox{for all } t>0
  \eas
  with $c_3\equiv c_3(m,p,s):=(2c_2)^\frac{p+m-1}{p\iota}$, whence (\ref{1.3}) implies that
  \bas
	\frac{d}{dt} \io u^p
	+ \frac{c_1}{c_3 K^\frac{(p+m-1)(1-\iota)}{\iota}} \cdot \bigg\{ \io u^p \bigg\}^\frac{p+m-1}{p\iota}
	\le \frac{c_1 c_3 K^\frac{p+m-1}{\iota}}{c_3 K^\frac{(p+m-1)(1-\iota)}{\iota}}
	= c_1 K^{p+m-1}
	\qquad \mbox{for all } t>0.
  \eas
  In view of Lemma \ref{lem76}, this firstly entails that whenever $T>0$,
  \be{1.4}
	\io u^p(\cdot,t) \le A t^{-\frac{p\iota}{p+m-1-p\iota}}
	\qquad \mbox{for all } t\in (0,T],
  \ee
  where
  \bas
	& & \hs{-20mm}
	A\equiv A(m,p,s,T) \nn\\
	&:=& \max \Bigg\{
	\bigg( \frac{2p\iota}{p+m-1-p\iota} \cdot \frac{c_3 K^\frac{(p+m-1)(1-\iota)}{\iota}}{c_1} \bigg)^\frac{p\iota}{p+m-1-p\iota} \, , \,  \\
	& & \hs{30mm}
	\bigg( 2c_1 K^{p+m-1} \cdot \frac{c_3 K^\frac{(p+m-1)(1-\iota)}{\iota}}{c_1} T^\frac{p+m-1}{p+m-1-p\iota} \bigg)^\frac{p\iota}{p+m-1}
		\Bigg\} \nn\\
	&=& \max \Bigg\{
	\bigg( \frac{2p\iota c_3}{(p+m-1-p\iota)c_1} \bigg)^\frac{p\iota}{p+m-1-p\iota} \cdot K^\frac{p(p+m-1)(1-\iota)}{p+m-1-p\iota} \, , \,
	(2c_3)^\frac{p\iota}{p+m-1} T^\frac{p\iota}{p+m-1-p\iota} K^p \
	\Bigg\} \\
	&\le& c_4 K^\frac{p(p+m-1)(1-\iota)}{p+m-1-p\iota}
	+ c_4 T^\frac{p\iota}{p+m-1-p\iota} K^p
  \eas
  with $c_4\equiv c_4(m,p,s):=\big(\frac{2p\iota c_3}{(p+m-1-p\iota)c_1}\big)^\frac{p\iota}{p+m-1-p\iota} + (2c_3)^\frac{p\iota}{p+m-1}$.
  Choosing $T:=K^{1-m}$ here, and using that then
  $T^\frac{p\iota}{p+m-1-p\iota} K^p = K^\frac{p(p+m-1)(1-\iota)}{p+m-1-p\iota}$, from (\ref{1.4}) we obtain that
  \be{1.5}
	\io u^p(\cdot,t) \le 2c_4 K^\frac{p(p+m-1)(1-\iota)}{p+m-1-p\iota} t^{-\frac{p\iota}{p+m-1-p\iota}}
	\qquad \mbox{for all } t\in (0,K^{1-m}]
  \ee
  and hence, in particular,
  \bas
	\io u^p(\cdot,K^{1-m}) \le 2c_4 K^\frac{p(p+m-1)(1-\iota)}{p+m-1-p\iota} \cdot K^{-\frac{p\iota}{p+m-1-p\iota} \cdot (1-m)}
	= 2c_4 K^p.
  \eas
  As $\frac{d}{dt} \io u^p \le 0$ for all $t>0$ by (\ref{1.3}), the latter ensures that
  \bas
	\io u^p(\cdot,t) \le 2c_4 K^p
	\qquad \mbox{for all } t>K^{1-m},
  \eas
  which combined with (\ref{1.5}) entails that
  \bas
	\io u^p(\cdot,t)
	\le 2c_4 K^\frac{p(p+m-1)(1-\iota)}{p+m-1-p\iota} t^{-\frac{p\iota}{p+m-1-p\iota}} + 2c_4 K^p
	\qquad \mbox{for all } t>0.
  \eas
  Since straightforward simplification using (\ref{1.33}) shows that
  \bas
	p(p+m-1)(1-\iota)=\frac{(p+m-1)s \cdot [2p+n(m-1)]}{(2-n)s+n(p+m-1)}
  \eas
  and
  \bas
	p+m-1-p\iota = \frac{(p+m-1) \cdot [2s+n(m-1)]}{(2-n)s+n(p+m-1)},
  \eas
  this readily establishes (\ref{1.2}) with $\Gamma_1(m,p,s):=(2c_4)^\frac{1}{p}$.
\qed
In further developing this in Lemma \ref{lem2} into an $L^\infty$-estimate of the form claimed in Theorem \ref{theo67},
similar to the approach in \cite{grillo_jde2013} we will pursue a Moser-type recursive strategy. In contrast
to the iteration created to derive upper estimates for the neigboring case $m>1$ in \cite{grillo_jde2013},
however, our argument will avoid the use of different starting points in each step,
but rather consider fixed time intervals at the price of an inclusion of suitably chosen time-dependent weight functions.
A core motivation for this technical difference originates from the ambition to design a framework which upon merely slight
adaptation can be revisited in Section \ref{sect3} to derive uniform lower bounds as a major objective of this study.\abs
In order to keep our presentation below as slim as possible, let us separately state an elementary observation on inhomogeneous
linear difference inequalities already here:
\begin{lem}\label{lem78}
  Let $(z_k)_{k\ge 0} \subset\R$ and $(h_k)_{k\ge 1} \subset \R$ be such that
  \bas
	z_k \le 2z_{k-1} + h_k
	\qquad \mbox{for all } k\ge 1.
  \eas
  Then
  \bas
	z_k \le 2^k z_0 + \sum_{j=1}^k 2^{k-j} h_j
	\qquad \mbox{for all } k\ge 0.
  \eas
\end{lem}
\proof
  This can be verified by straightforward induction.
\qed
In fact, we are now prepared to achieve the essential part of the $L^\infty$ estimate in Theorem \ref{theo67};
the restriction to a setting of conveniently normalized initial data undertaken here will subsequently be removed
by means of a scaling argument.
\begin{lem}\label{lem2}
  Let $n\ge 1, m\in (-\infty,1)$ and $s\ge 1$ be such that
  \be{2.001}
	s>\frac{n(1-m)}{2}.
  \ee
  Then there exists $\Gamma_2=\Gamma_2(m,s)>0$ such that if $u_0\in C^0(\bom)$ is positive in $\bom$ with
  \be{2.01}
	\|u_0\|_{L^s(\Om)} = 1,
  \ee
  and $u\in C^0(\bom\times [0,\infty)) \cap C^{2,1}(\bom\times (0,\infty))$ is positive in $\bom\times [0,\infty)$ and such that
  (\ref{0}) is solved classically, then
  \be{2.1}
	\|u(\cdot,t)\|_{L^\infty(\Om)}
	\le \Gamma_2 t^{-\frac{n}{2s+n(m-1)}} + \Gamma_2
	\qquad \mbox{for all } t>0.
  \ee
\end{lem}
\proof
  We pick any $p_0>s$ such that
  \be{2.110}
	p_0>1-m
	\qquad \mbox{and} \qquad
	p_0>\frac{(n+2)(1-m)}{4},
  \ee
  and given $u_0$ and $u$ with the listed properties we observe that in line with (\ref{2.01}),
  Lemma \ref{lem1} shows that
  \bea{2.112}
	1 + \io u^{p_0}(\cdot,t)
	&\le& 1 + \Big\{ \Gamma_1 t^{-\frac{n(1-\frac{s}{p_0})}{2s+n(m-1)}} + \Gamma_1
		\Big\}^{p_0} \nn\\
	&\le& 1
	+ (2\Gamma_1)^{p_0} t^{-\frac{n(p_0-s)}{2s+n(m-1)}}
	+ (2\Gamma_1)^{p_0} \nn\\
	&\le& c_1 t^{-\frac{n(p_0-s)}{2s+n(m-1)}}
	\qquad \mbox{for all } t\in (0,1),
  \eea
  where $c_1\equiv c_1(m,s):=1+(2\Gamma_1)^{p_0}$ with $\Gamma_1=\Gamma_1(m,p_0,s)$ as provided there.\abs
  To derive (\ref{2.1}) from this, we let $\gamma_0:=0$ and, for positive integers $k$,
  \be{2.113}
	p_k:=2^k p_0 + (2^k-1) \cdot \frac{n(m-1)}{2}
	\qquad \mbox{and} \qquad
	\gamma_k:=(2^k-1) \cdot \frac{n}{2}
  \ee
  noting that then, as $p_0>s>\frac{n(1-m)}{2}$ by (\ref{2.001}), $(p_k)_{k\ge 0}$ and $(\gamma_k)_{k\ge 0}$ are increasing with
  \be{2.114}
	p_k=2p_{k-1} + \frac{n(m-1)}{2}
	\quad \mbox{and} \quad
	\gamma_k = 2\gamma_{k-1} + \frac{n}{2}
	\qquad \mbox{for all } k\ge 1.
  \ee
  Moreover, for fixed $t_0 \in (0,1]$
  and $k\ge 0$ we set
  \be{2.115}
	M_k\equiv M_k(t_0):=\sup_{t\in (\frac{t_0}{2},1]} \Bigg\{
		\Big( t-\frac{t_0}{2}\Big)^{\gamma_k} \cdot \bigg\{ 1 + \io u^{p_k}(\cdot,t) \bigg\} \Bigg\},
  \ee
  and observe that according to (\ref{2.112}),
  \bea{2.116}
	M_0
	&=& 1 + \sup_{t\in (\frac{t_0}{2},1]} \io u^{p_0}(\cdot,t) \nn\\
	&\le& 1 + \io u^{p_0}\Big(\cdot,\frac{t_0}{2}\Big) \nn\\
	&\le& c_2 t_0^{-\frac{n(p_0-s)}{2s+n(m-1)}}
  \eea
  with $c_2\equiv c_2(m,s):=2^\frac{n(p_0-s)}{2s+n(m-1)} c_1$, because $0<t\mapsto \io u^{p_0}(\cdot,t)$ is nonincreasing
  by (\ref{1.3}) and the fact that $p_0>s\ge 1$.
  To estimate $M_k$ for $k\ge 1$, we first use (\ref{0}) in estimating
  \be{2.117}
	\frac{d}{dt} \io u^{p_k}
	= - \frac{4p_k(p_k-1)}{(p_k+m-1)^2} \io \big| \na u^\frac{p_k+m-1}{2}\big|^2
	\le -c_3 \io \big|\na u^\frac{p_k+m-1}{2}\big|^2
	\qquad \mbox{for all } t>0,
  \ee
  with $c_3\equiv c_3(m):=\inf_{p\ge p_1} \frac{4p(p-1)}{(p+m-1)^2}$ being positive, because
  $p_1+m-1=2p_0+\frac{(n+2)(m-1)}{2} >0$ by (\ref{2.113}) and (\ref{2.110}).
  Now the integral on right-hand side can be bounded from below by introducing
  \be{2.1133}
	q_k:=\frac{2p_{k-1}}{p_k+m-1},
	\qquad k\ge 1,
  \ee
  so that, by (\ref{2.113}) and (\ref{2.114}),
  \be{2.1134}
	q_k
	= \frac{2^k p_0 + (2^k-2) \cdot \frac{n(m-1)}{2}}{2^k p_0 + (2^k-1) \cdot \frac{n(m-1)}{2} + m-1}
	> \frac{2^k p_0 + (2^k-1) \cdot\frac{n(m-1)}{2}}{2^k p_0 + (2^k-1) \cdot \frac{n(m-1)}{2} + m-1}
	> 1
	\qquad \mbox{for all } k\ge 1.
  \ee
  Apart from that, we note that the inequality $m<1$ along with (\ref{2.113}) ensures that
  \bas
	\frac{2p_k}{p_k+m-1}
	\le \frac{2p_1}{p_1+m-1}
	= \frac{4p_0+n(m-1)}{2p_0+\frac{(n+2)(m-1)}{2}}
	\qquad \mbox{for all } k\ge 1,
  \eas
  so that since
  \bas
	& & \hs{-18mm}
	\frac{2p_0+\frac{(n+2)(m-1)}{2}}{4p_0+n(m-1)} - \frac{n-2}{2n} \\
	&=& \frac{1}{2n\cdot [4p_0+n(m-1)]} \cdot \bigg\{
	2n\cdot\Big\{ 2p_0+\frac{(n+2)(m-1)}{2}\Big\} - (n-2)\cdot\big\{ 4p_0+n(m-1)\big\} \bigg\} \\
	&=& \frac{1}{2n\cdot [4p_0+n(m-1)]} \cdot \big\{
	4np_0+n(n+2)(m-1) - 4(n-2)p_0 -n(n-2)(m-1)\big\} \\
	&=& \frac{1}{2n\cdot [4p_0+n(m-1)]} \cdot \big\{
	8p_0+4n(m-1)\big\} \\
	&>& 0
  \eas
  according to the fact that $p_0>\frac{n(1-m)}{2}$, it follows that
  \be{2.1135}
	\sup_{k\ge 1} \frac{2p_k}{p_k+m-1} < \frac{2n}{(n-2)_+}.
  \ee
  As clearly $\frac{2p_{k-1}}{p_k+m-1} < \frac{2p_k}{p_k+m-1}$ for all $k\ge 1$, based on (\ref{2.1134}) and (\ref{2.1135})
  we may combine the Gagliardo-Nirenberg inequality with the H\"older inequality to find $c_4=c_4(m,s)>0$ such that for each
  $k\ge 1$ and any $\vp\in W^{1,2}(\Om)$,
  \bea{2.118}
	c_4 \|\vp\|_{L^\frac{2p_k}{p_k+m-1}(\Om)}^{\frac{n+2}{n}\cdot \frac{2p_k}{p_k+m-1}}
	\le c_3 \|\na\vp\|_{L^2(\Om)}^{\frac{n+2}{n} \cdot \frac{2p_k}{p_k+m-1} \cdot \iota_k}
		\|\vp\|_{L^\frac{2p_{k-1}}{p_k+m-1}(\Om)}^{\frac{n+2}{n} \cdot \frac{2p_k}{p_k+m-1} \cdot (1-\iota_k)}
		+ c_3 \|\vp\|_{L^{q_k}(\Om)}^{\frac{n+2}{n}\cdot\frac{2p_k}{p_k+m-1}},
  \eea
  with $\iota_k\in (0,1)$ satisfying
  \bas
	\iota_k=\frac{n(p_k+m-1)(p_k-p_{k-1})}{p_k\cdot [(2-n)p_{k-1}+n(p_k+m-1)]}
	\qquad \mbox{and} \qquad
	1-\iota_k=\frac{p_{k-1} \cdot [2p_k+n(m-1)]}{p_k\cdot [(2-n)p_{k-1}+n(p_k+m-1)]}.
  \eas
  Here, in line with (\ref{2.114}) we can simplify according to
  \bas
	\frac{n+2}{n} \cdot \frac{2p_k}{p_k+m-1} \cdot \iota_k
	&=& 2(n+2)\cdot\frac{p_k-p_{k-1}}{(2-n)p_{k-1}+n(p_k+m-1)} \\
	&=& 2(n+2)\cdot \frac{p_{k-1}+\frac{n(m-1)}{2}}{(2-n)p_{k-1} + n(2p_{k-1}+ \frac{(n+2)(m-1)}{2})}
	= 2
  \eas
  and
  \bas
	\frac{n+2}{n} \cdot \frac{2p_k}{p_k+m-1} \cdot (1-\iota_k)
	&=& \frac{2p_{k-1}}{p_k+m-1} \cdot \frac{n+2}{n} \cdot \frac{p_k}{p_{k-1}} \cdot (1-\iota_k) \\
	&=& \frac{2p_{k-1}}{p_k+m-1} \cdot \frac{n+2}{n} \cdot \frac{2p_k + n(m-1)}{(2-n)p_{k-1} +n(p_k+m-1)} \\
	&=& \frac{2p_{k-1}}{p_k+m-1} \cdot \frac{n+2}{n} \cdot \frac{4p_{k-1} +2n(m-1)}{(2-n)p_{k-1} +n(2p_{k-1}+\frac{(n+2)(m-1)}{2})} \\
	&=& \frac{2p_{k-1}}{p_k+m-1} \cdot \frac{4}{n}
  \eas
  for all $k\ge 1$, whence (\ref{2.118}) shows that
  \bas
	c_4 \cdot \bigg\{ \io u^{p_k} \bigg\}^\frac{n+2}{n}
	&=& c_4 \big\| u^\frac{p_k+m-1}{2} \big\|_{L^\frac{2p_k}{p_k+m-1}(\Om)}^{\frac{n+2}{n}\cdot\frac{2p_k}{p_k+m-1}} \\
	&\le& c_3 \big\| \na u^\frac{p_k+m-1}{2} \big\|_{L^2(\Om)}^2
		\big\| u^\frac{p_k+m-1}{2}\big\|_{L^\frac{2p_{k-1}}{p_k+m-1}(\Om)}^{\frac{2p_{k-1}}{p_k+m-1} \cdot \frac{4}{n}}
	+ c_3 \big\| u^\frac{p_k+m-1}{2} \big\|_{L^{q_k}(\Om)}^{\frac{n+2}{n}\cdot\frac{2p_k}{p_k+m-1}} \\
	&=& c_3 \big\| \na u^\frac{p_k+m-1}{2}\big\|_{L^2(\Om)}^2 \cdot
		\bigg\{ \io u^{p_{k-1}} \bigg\}^\frac{4}{n}
	+ c_3 \|u\|_{L^\frac{(p_k+m-1)q_k}{2}(\Om)}^\frac{(n+2)p_k}{n}
  \eas
  and thus, by (\ref{2.117}),
  \be{2.119}
	\frac{d}{dt} \io u^{p_k} + c_4\cdot\bigg\{ \io u^{p_{k-1}} \bigg\}^{-\frac{4}{n}} \cdot
		\bigg\{ \io u^{p_k} \bigg\}^\frac{n+2}{n}
	\le c_3\cdot \bigg\{ \io u^{p_{k-1}} \bigg\}^{-\frac{4}{n}} \cdot \|u\|_{L^\frac{(p_k+m-1)q_k}{2}(\Om)}^\frac{(n+2)p_k}{n}
  \ee
  for all $t>0$ and $k\ge 1$.
  Here, using that
  \be{2.120}
	1 + \io u^{p_{k-1}} \le \Big(t-\frac{t_0}{2}\Big)^{-\gamma_{k-1}} M_{k-1}
	\qquad \mbox{for all $t\in (\frac{t_0}{2},1)$ and } k\ge 1
  \ee
  due to (\ref{2.115}), again thanks to (\ref{2.114}) we obtain that
  \bas
	\bigg\{ \io u^{p_{k-1}} \bigg\}^{-\frac{4}{n}} \cdot \|u\|_{L^\frac{(p_k+m-1)q_k}{2}(\Om)}^\frac{(n+2)p_k}{n}
	&=& \bigg\{ \io u^{p_{k-1}}\bigg\}^{-\frac{4}{n}} \cdot \|u\|_{L^{p_{k-1}}(\Om)}^\frac{(n+2)p_k}{n} \nn\\
	&=& \bigg\{ \io u^{p_{k-1}} \bigg\}^{-\frac{4}{n}+\frac{(n+2)p_k}{np_{k-1}}} \nn\\
	&=& \bigg\{ \io u^{p_{k-1}} \bigg\}^{-\frac{4}{n}+\frac{n+2}{n}\cdot \frac{2p_{k-1}+\frac{n(m-1)}{2}}{p_{k-1}}} \nn\\
	&=& \bigg\{ \io u^{p_{k-1}} \bigg\}^{2+\frac{(n+2)(m-1)}{2p_{k-1}}} \nn\\
	&\le& \bigg\{ 1 + \io u^{p_{k-1}} \bigg\}^2 \nn\\
	&\le& \Big(t-\frac{t_0}{2}\Big)^{-2\gamma_{k-1}} M_{k-1}^2
	\qquad \mbox{for all $t\in (\frac{t_0}{2},1)$ and } k\ge 1.
  \eas
  In conjunction with (\ref{2.119}) and again (\ref{2.120}), this implies that
  for all $t\in (\frac{t_0}{2},1)$ and $k\ge 1$,
  \bas
	\frac{d}{dt} \io u^{p_k}
	+ c_4\cdot\Big(t-\frac{t_0}{2}\Big)^\frac{4\gamma_{k-1}}{n} M_{k-1}^{-\frac{4}{n}} \cdot
		\bigg\{ \io u^{p_k} \bigg\}^\frac{n+2}{n}
	\le c_3 \cdot \Big(t-\frac{t_0}{2}\Big)^{-2\gamma_{k-1}} M_{k-1}^2,
  \eas
  whence observing that for each fixed $k\ge 1$, writing
  $a:=c_4 M_{k-1}^{-\frac{4}{n}}, b:= c_3 M_{k-1}^2, \al:=\frac{n+2}{n}, \lam:=\frac{4\gamma_{k-1}}{n}$ and
  $\mu:=2\gamma_{k-1}$ we have
  \bas
	\frac{\lam+1}{\al-1}
	= \frac{\frac{4\gamma_{k-1}}{n}+1}{\frac{2}{n}}
	= 2\gamma_{k-1} + \frac{n}{2}
	= \gamma_k
  \eas
  and
  \bas
	\frac{\al+\lam}{\al-1}-\mu = \frac{n+2}{2}+ \frac{\frac{4\gamma_{k-1}}{n}}{\frac{2}{n}} - 2\gamma_{k-1}
	= \frac{n+2}{2}\ge 0
  \eas
  by (\ref{2.113}), from Lemma \ref{lem76} we infer that
  \be{2.121}
	\io u^{p_k}(\cdot,t) \le A_k \cdot\Big(t-\frac{t_0}{2}\Big)^{-\gamma_k}
	\qquad \mbox{for all $t\in (\frac{t_0}{2},1]$ and } k\ge 1,
  \ee
  where
  \bas
	A_k &:=&\max \Bigg\{ \bigg( \frac{2\gamma_k}{c_4 M_{k-1}^{-\frac{4}{n}}}\bigg)^\frac{n}{2} \, , \,
		\bigg( \frac{2c_3 M_{k-1}^2}{c_4 M_{k-1}^{-\frac{4}{n}}}\bigg)^\frac{n}{n+2} \Bigg\}\\
	&=& \max \bigg\{ \Big(\frac{2}{c_4}\Big)^\frac{n}{2} \cdot \gamma_k^\frac{n}{2} \, , \, \Big(\frac{2c_3}{c_4}\Big)^\frac{n}{n+2}
       		\bigg\} \cdot M_{k-1}^2.
  \eas
  By means of the rough estimate
  \bas
	\max \bigg\{ \Big(\frac{2}{c_4}\Big)^\frac{n}{2} \cdot \gamma_k^\frac{n}{2} \, , \, \Big(\frac{2c_3}{c_4}
      \Big)^\frac{n}{n+2}
      		\bigg\}
	&\le& \Big(\frac{2}{c_4}\Big)^\frac{n}{2} \gamma_k^\frac{n}{2} + \Big(\frac{2c_3}{c_4}
      \Big)^\frac{n}{n+2}
     \\
	&=& \Big(\frac{2}{c_4}\Big)^\frac{n}{2} \cdot \Big\{ (2^k-1)\cdot\frac{n}{2}\Big\}^\frac{n}{2} + \Big(\frac{2c_3}{c_4}
      \Big)^\frac{n}{n+2}
       \\
	&\le& \Big( \frac{n}{c_4}\Big)^\frac{n}{2} \cdot 2^\frac{nk}{2} + \Big(\frac{2c_3}{c_4}
      \Big)^\frac{n}{n+2}
      \\
	&\le& \Big\{ \Big(1+\frac{n}{c_4}\Big)^\frac{n}{2} \cdot 2^\frac{n}{2} + \Big(1+\frac{2c_3}{c_4}\Big)^\frac{n}{n+2} \Big\}^k,
  \eas
  from (\ref{2.121}) we thus obtain that in line with (\ref{2.115}),
  \bas
	M_k \le c_5^k M_{k-1}^2
  \eas
  and hence
  \bas
	\ln M_k \le 2\ln M_{k-1} + k\ln c_5
  \eas
  for all $k\ge 1$,
  where $c_5\equiv c_5(m,s):=1+\Big(1+\frac{n}{c_4}\Big)^\frac{n}{2} \cdot 2^\frac{n}{2} + \Big(1+\frac{2c_3}{c_4}\Big)^\frac{n}{n+2}$.
  An application of Lemma \ref{lem78} using the fact that $\sum_{j=1}^k j\cdot 2^{k-j} = 2^{k+1} - k - 2$ for all $k\ge 1$
  therefore shows that
  \bas
	M_k \le c_5^{2^{k+1}-k-2} M_0^{2^k}
	\le c_5^{2^{k+1}} M_0^{2^k}
	\qquad \mbox{for all } k\ge 1,
  \eas
  so that
  \bas
	\|u(\cdot,t)\|_{L^{p_k}(\Om)}
	&\le& \Big\{ \Big(t-\frac{t_0}{2}\Big)^{-\gamma_k} M_k\Big\}^\frac{1}{p_k} \nn\\
	&\le& c_5^\frac{2^{k+1}}{p_k} M_0^\frac{2^k}{p_k} \cdot \Big(t-\frac{t_0}{2}\Big)^{-\frac{\gamma_k}{p_k}}
	\qquad \mbox{for all $t\in (\frac{t_0}{2},1]$ and } k\ge 1.
  \eas
  Since (\ref{2.113}) ensures that
  \bas
	\frac{2^k}{p_k}
	= \frac{2^k}{2^k p_0 + (2^k-1)\cdot\frac{n(m-1)}{2}}
	\to \frac{1}{p_0+\frac{n(m-1)}{2}}
	= \frac{2}{2p_0+n(m-1)}
  \eas
  and
  \bas
	\frac{\gamma_k}{p_k}
	= \frac{(2^k-1)\cdot\frac{n}{2}}{2^k p_0 + (2^k-1)\cdot\frac{n(m-1)}{2}}
	\to \frac{\frac{n}{2}}{p_0+\frac{n(m-1)}{2}}
	= \frac{n}{2p_0+n(m-1)}
  \eas
  as $k\to\infty$, in conjunction with (\ref{2.116}) this implies that
  for all $t\in (\frac{t_0}{2},1]$,
  \bea{2.999}
	\|u(\cdot,t)\|_{L^\infty(\Om)}
	&\le& c_5^\frac{4}{2p_0+n(m-1)} M_0^\frac{2}{2p_0+n(m-1)} \cdot \Big(t-\frac{t_0}{2}\Big)^{-\frac{n}{2p_0+n(m-1)}} \nn\\
	&\le& c_2^\frac{2}{2p_0+n(m-1)} c_5^\frac{4}{2p_0+n(m-1)} t_0^{-\frac{2n(p_0-s)}{[2p_0+n(m-1)]\cdot [ 2s+n(m-1)]}}
		\cdot \Big(t-\frac{t_0}{2}\Big)^{-\frac{n}{2p_0+n(m-1)}}.
  \eea
  Observing that
  \bas
	& & \hs{-20mm}
	\frac{2n(p_0-s)}{[2p_0+n(m-1)]\cdot [ 2s+n(m-1)]} + \frac{n}{2p_0+n(m-1)} \nn\\
	&=& \frac{n}{2p_0+n(m-1)} \cdot \frac{2(p_0-s) + [2s+n(m-1)]}{2s+n(m-1)} \\
	&=& \frac{n}{2p_0+n(m-1)} \cdot \frac{2p_0+n(m-1)}{2s+n(m-1)}
	= \frac{n}{2s+n(m-1)},
  \eas
  upon evaluating (\ref{2.999}) at $t=t_0$ we obtain that
  \bas
	\|u(\cdot,t_0)\|_{L^\infty(\Om)}
	\le c_6 t_0^{-\frac{n}{2s+n(m-1)}}
	\qquad \mbox{for all } t_0\in (0,1],
  \eas
  where $c_6\equiv c_6(m,s):=2^\frac{n}{2p_0+n(m-1)} c_2^\frac{2}{2p_0+n(m-1)} c_5^\frac{4}{2p_0+n(m-1)}$.
  Since $0<t\mapsto \|u(\cdot,t)\|_{L^\infty(\Om)}$ is nonincreasing by the maximum principle, this especially entails that
  \bas
	\|u(\cdot,t)\|_{L^\infty(\Om)}
	\le \|u(\cdot,1)\|_{L^\infty(\Om)}
	\le c_6
	\qquad \mbox{for all } t>1,
  \eas
  so that the inequality in (\ref{2.1}) follows for arbitrary $t>0$ if we let $\Gamma_2(m,s):=c_6$.
\qed
Indeed, for smooth solutions the above quite immediately implies the following.
\begin{lem}\label{lem3}
  Let $n\ge 1, m\in (-\infty,1)$ and $s\ge 1$ be such that $s>\frac{n(1-m)}{2}$.
  Then whenever $u_0\in C^0(\bom)$ is positive and $u\in C^0(\bom\times [0,\infty)) \cap C^{2,1}(\bom\times (0,\infty))$ is
  positive in $\bom\times [0,\infty)$ and solves (\ref{0}), the estimate in (\ref{2.1}) is valid for all $t>0$.
\end{lem}
\proof
  We let $K:=\|u_0\|_{L^s(\Om)}$ and $v(x,\tau):=\frac{1}{K}\cdot u(x,K^{1-m}\tau)$ for $x\in\bom$ and $\tau\ge 0$.
  Then $v_0:=v(\cdot,0)$ is positive in $\bom$ and satisfies $\|v_0\|_{L^s(\Om)}=1$, so that since
  $v_\tau=\na\cdot (v^{m-1}\na v)$ in $\Om\times (0,\infty)$ with $\frac{\pa v}{\pa\nu}=0$ on $\pO\times (0,\infty)$
  by (\ref{0}), Lemma \ref{lem2} applies so as to ensure that
  \bas
	\|v(\cdot,\tau)\|_{L^\infty(\Om)}
	\le \Gamma_2 \tau^{-\frac{n}{2s+n(m-1)}} + \Gamma_2
	\qquad \mbox{for all } \tau>0.
  \eas
  As thus
  \bas
	\|u(\cdot,t)\|_{L^\infty(\Om)}
	&=& K \|v(\cdot,K^{m-1}t)\|_{L^\infty(\Om)} \\
	&\le& K\Gamma_2 \cdot (K^{m-1} t)^{-\frac{n}{2s+n(m-1)}} + K\Gamma_2 \\
	&=& \Gamma_2 K^{1-\frac{n(m-1)}{2s+n(m-1)}} t^{-\frac{n}{2s+n(m-1)}} + K\Gamma_2
	\qquad \mbox{for all } t>0,
  \eas
  this readily yields the claim.
\qed
\mysection{Lower bounds. Positivity}\label{sect3}
The analysis in this section will be centered around tracking the evolution of $\io u^{-p}$ for large $p>0$, again resorting to
smooth trajectories first.
At the initial level of a corresponding Moser iteration, however, we here need to adequatly cope with the circumstance that,
unlike in Lemma \ref{lem1}, neither of these integrals can be expected to satisfy appropriate upper bounds near $t=0$ when
finally the inclusion of merely nonnegative initial data, possibly with compact support in $\Om$, is striven for.
An alternative type of appropriate information at the lower end of a Gagliardo-Nirenberg interpolation in the style of that
from Lemma \ref{lem1} is prepared by a variant of the classical Poincar\'e inequality which, in fact, forms
the origin of the factors $\zeta^\star$ that finally appear in Theorem \ref{theo68}.
Although the following statement in this direction might be well-known, since we could not find a precise reference
we include a brief proof in an appendix below, especially in order to let the origin of the dimension dependence
in (\ref{11.2})-(\ref{zeta}) become transparent.
\begin{lem}\label{lem11}
  There exists $\kappa=\kappa(\Om)>0$ such that whenever $\eta\in (0,|\Om|]$ and $\vp\in W^{1,\infty}(\Om)$ are such that
  \be{11.1}
	\big| \{\vp=0\} \big| \ge \eta,
  \ee
  it follows that
  \be{11.2}
	\io \vp^2 \le \kappa \cdot \zeta_1(\eta) \io |\na\vp|^2,
  \ee
  where $\zeta_1$ is as accordingly defined through (\ref{zeta}).
\end{lem}
We can thereby establish the following counterpart of Lemma \ref{lem1}.
\begin{lem}\label{lem12}
  Let $n\ge 1, m\in (-\infty,1), q\in (0,1]$ and $r\ge 1$ be such that $q<r$, and let $p>1$.
  Then there exists $\Gamma_3=\Gamma_3(m,p,q,r)>0$ such that if $u_0\in C^0(\bom)$ is positive in $\bom$ and
  $u\in C^0(\bom\times [0,\infty)) \cap C^{2,1}(\bom\times (0,\infty))$ is positive in $\bom\times (0,\infty)$ and solves (\ref{0})
  in the classical sense, then
  \be{12.1}
	\Big\| \frac{1}{u(\cdot,t)} \Big\|_{L^p(\Om)}
	\le \Gamma_3 \cdot \zeta_{\frac{qr}{r-q}} \Big( \mbox{$\frac{\|u_0\|_{L^q(\Om)}}{\|u_0\|_{L^r(\Om)}}$} \Big) \cdot t^{-\frac{1}{1-m}}
	+ \Gamma_3 \cdot \frac{1}{\|u_0\|_{L^q(\Om)}}
	\qquad \mbox{for all } t>0,
  \ee
  where $\zeta_{\frac{qr}{r-q}}$ is as in (\ref{zeta}).
\end{lem}
\proof
  Integrating by parts in (\ref{0}) we see that
  \bea{12.11}
	\frac{d}{dt} \io u^{-p}
	&=& -p \io u^{-p-1} \na\cdot (u^{m-1}\na u)
	= - p(p+1) \io u^{-p+m-3} |\na u|^2 \nn\\
	&=& - c_1 \io \big|\na u^{-\frac{p-m+1}{2}}\big|^2
	\qquad \mbox{for all } t>0,
  \eea
  where $c_1\equiv c_1(m,p):=\frac{4p(p+1)}{(p-m+1)^2}$ is well-defined and positive, because $p>0>m-1$.
  As from (\ref{0}) we furthermore know that $\frac{d}{st} \io u^r \le 0$ and $\frac{d}{dt} \io u^q \ge 0$ due to the inequalities
  $r\ge 1$ and $q\le 1$, abbreviating $K_\sigma:=\|u_0\|_{L^\sig(\Om)}$ for $\sig>0$ we can use the H\"older inequality to
  see that for all $\del>0$ and $t>0$,
  \bas
	K_q^q
	= \io u_0^q
	\le \io u^q
	= \int_{\{u\le\del\}} u^q + \int_{\{u>\del\}} u^q
	\le \del^q |\Om|
	+ \bigg\{ \io u^r\bigg\}^\frac{q}{r} \cdot \big|\{u>\del\}\big|^\frac{r-q}{r},
  \eas
  which upon choosing
  \be{12.2}
	\del:=\frac{K_q}{(2|\Om|)^\frac{1}{q}}
  \ee
  implies that
  \bas
	\frac{1}{2} K_q^q \le K_r^q \cdot \big| \{u>\del\}\big|^\frac{r-q}{r}
	\qquad \mbox{for all } t>0,
  \eas
  that is,
  \be{12.3}
	\big| \{u>\del\}\big|
	\ge \eta:=\Big(\frac{K_q^q}{2K_r^q}\Big)^\frac{r}{r-q}
	\qquad \mbox{for all } t>0.
  \ee
  For fixed $t>0$, the function $\vp:=\big(u^{-\frac{p-m+1}{2}} - \del^{-\frac{p-m+1}{2}}\big)_+(\cdot,t)$ thus lies in
  $W^{1,\infty}(\Om)$ and satisfies $\big| \{\vp=0\}\big| = \big| \{ u(\cdot,t)>\del\}\big| \ge \eta$, so that Lemma \ref{lem11}
  applies so as to warrant that with $\kappa=\kappa(\Om)$ as found there,
  \bas
	\io \big(u^{-\frac{p-m+1}{2}} - \del^{-\frac{p-m+1}{2}}\big)_+^2
	&\le& \kappa \cdot \zeta_1(\eta) \cdot \io \Big| \na \big(u^{-\frac{p-m+1}{2}} - \del^{-\frac{p-m+1}{2}}\big)_+ \Big|^2 \\
	&=& \kappa\cdot\zeta_1(\eta) \cdot \int_{\{u<\del\}}
    	\Big| \na \big(u^{-\frac{p-m+1}{2}} - \del^{-\frac{p-m+1}{2}} \big)\Big|^2 \\
	&=& \kappa\cdot\zeta_1(\eta) \cdot \int_{\{u<\del\}} \big|\na u^{-\frac{p-m+1}{2}}\big|^2.
  \eas
  Since $\frac{\xi^{-\gamma}}{\xi^{-\gamma}-\del^{-\gamma}} = \frac{1}{1-(\frac{\xi}{\del})^\gamma} \le \frac{1}{1-2^{-\gamma}}$
  for all $\xi\in (0,\frac{\del}{2})$ and $\gamma>0$, we therefore obtain that
  \bas
	\io u^{-(p-m+1)}
	&=& \int_{\{u<\frac{\del}{2}\}} \Big( \frac{u^{-\frac{p-m+1}{2}}}{u^{-\frac{p-m+1}{2}} - \del^\frac{p-m+1}{2}} \Big)^2
		\cdot \big( u^{-\frac{p-m+1}{2}} - \del^{-\frac{p-m+1}{2}} \big)^2
	+ \int_{\{u\ge\frac{\del}{2}\}} u^{-(p-m+1)} \\
	&\le& \frac{1}{(1-2^{-\frac{p-m+1}{2}})^2} \int_{\{u<\frac{\del}{2}\}}
		\big( u^{-\frac{p-m+1}{2}} - \del^{-\frac{p-m+1}{2}} \big)^2
	+ \Big(\frac{\del}{2}\Big)^{-(p-m+1)} |\Om| \\
	&\le& \frac{1}{(1-2^{-\frac{p-m+1}{2}})^2} \io \big( u^{-\frac{p-m+1}{2}} - \del^{-\frac{p-m+1}{2}} \big)_+^2
	+ \Big(\frac{\del}{2}\Big)^{-(p-m+1)} |\Om| \\
	&\le& c_2 \zeta_1(\eta) \io \big| \na u^{-\frac{p-m+1}{2}}\big|^2
	+ c_2 \del^{-(p-m+1)}
	\qquad \mbox{for all } t>0
  \eas
  with $c_2\equiv c_2(m,p):=\max \big\{ (1-2^{-\frac{p-m+1}{2}})^{-2} \kappa \, , \, 2^{p-m+1} |\Om| \big\}$.
  Relying on the inequality $p<p-m+1$ in estimating
  \bas
	\io u^{-p}
	\le \bigg\{ \io u^{-(p-m+1)} \bigg\}^\frac{p}{p-m+1} |\Om|^\frac{1-m}{p-m+1}
	\qquad \mbox{for all } t>0
  \eas
  by means of the H\"older inequality, from (\ref{12.11}) we hence infer upon writing
  $c_3\equiv c_3(p,m):=\frac{c_1}{c_2} (|\Om|+1)^{m-1}$ that
  \bas
	\frac{d}{dt} \io u^{-p}
	+ \frac{c_3}{\zeta_1(\eta)} \cdot \bigg\{ \io u^{-p}\bigg\}^\frac{p-m+1}{p}
	\le \frac{c_1 \del^{-(p-m+1)}}{\zeta_1(\eta)}
	\qquad \mbox{for all } t>0.
  \eas
  An application of Lemma \ref{lem75} to $a:=\frac{c_3}{\zeta_1(\eta)}, b:=\frac{c_1 \del^{-(p-m+1)}}{\zeta_1(\eta)}$
  and $\al:=\frac{p-m+1}{p}$ thus reveals that
  \bas	
	\io u^{-p} \le \Big(\frac{p\zeta_1(\eta)}{(1-m)c_3}\Big)^\frac{p}{1-m} t^{-\frac{p}{1-m}}
	+ \Big(\frac{c_1 \del^{-(p-m+1)}}{c_3}\Big)^\frac{p}{p-m+1}
  \eas
  and hence
  \bas
	\Big\| \frac{1}{u}\Big\|_{L^p(\Om)}
	\le 2^\frac{1}{p} \cdot \Big(\frac{p\zeta_1(\eta)}{(1-m)c_3}\Big)^\frac{1}{1-m} t^{-\frac{1}{1-m}}
	+ 2^\frac{1}{p} \Big(\frac{c_1 \del^{-(p-m+1)}}{c_3}\Big)^\frac{1}{p-m+1}
  \eas
  for all $t>0$.
  In view of our definitions of $\del$ and $\eta$ in (\ref{12.2}) and (\ref{12.3}), this establishes (\ref{12.1})
  if $\Gamma_3=\Gamma_3(m,p,q,r)$ is sufficiently large, because in line with (\ref{zeta}), letting
  $\eta_0:=\frac{\|u_0\|_{L^q(\Om)}}{\|u_0\|_{L^r(\Om)}}$ we have
  \bas
	\zeta_1(\eta)
	= \big(2^{-\frac{r}{r-q}} \eta_0^\frac{qr}{r-q}\big)^{-\frac{n-2}{n}}
	= 2^\frac{(n-2)r}{n(r-q)} \zeta_{\frac{qr}{r-q}}(\eta_0)
  \eas
  when $n\ge 3$, as well as
  \bas
	\zeta_1(\eta)
	= 1 + \ln_+ \Big( 2^\frac{r}{r-q} \eta_0^{-\frac{qr}{q-r}} \Big)
	\le 1 + \ln 2^\frac{r}{r-q} + \frac{qr}{q-r} \ln_+ \frac{1}{\eta_0}
  \eas
  if $n=2$.
\qed
In turning the above into the lower bound addressed in Theorem \ref{theo68}, we will make use of the following elementary ineqiality.
\begin{lem}\label{lem79}
  Let $X>0$ and $Y>0$. Then
  \bas
	\ln (X+Y) \le \ln_+ X + \ln_+ Y + \ln 2.
  \eas
\end{lem}
\proof
  Without loss of generality assuming that $X\ge Y$, we only need to observe that then
  $\ln (X+Y) \le \ln (2X) = \ln X + \ln 2 \le \ln_+ X + \ln 2 \le \ln_+ X + \ln_+ Y + \ln 2$
  by nonnegativity of $\ln_+$.
\qed
Forming a core part of the analysis in this paper, a second Moser iteration uses Lemma \ref{lem12} as a starting point
for a recursive reasoning toward lower estimates for smooth solutions, depending on the initial data
only through the quantities in (\ref{68.2}):
\begin{lem}\label{lem13}
  Let $n\ge 1, m\in (-\infty,1), q\in (0,1]$ and $r\ge 1$ be such that $q<r$.
  Then there exists $\Gamma_4=\Gamma_4(m,q,r)>0$ such that if $0<u_0\in C^0(\bom)$ and
  $0<u\in C^0(\bom\times [0,\infty)) \cap C^{2,1}(\bom\times (0,\infty))$ are such that (\ref{0}) is solved
  in the classical sense, then with $\zeta_{\frac{qr}{r-q}}$ taken from (\ref{zeta}), we have
  \be{13.1}
	\Big\| \frac{1}{u(\cdot,t)} \Big\|_{L^\infty(\Om)}
	\le \Gamma_4 \cdot \zeta_{\frac{qr}{r-q}}^{1+n(1-m)} \Big( \mbox{$\frac{\|u_0\|_{L^q(\Om)}}{\|u_0\|_{L^r(\Om)}}$} \Big)
	\cdot t^{-\frac{1}{1-m}}
	+ \Gamma_4 \cdot \zeta_{\frac{qr}{r-q}}^{n(1-m)} \Big( \mbox{$\frac{\|u_0\|_{L^q(\Om)}}{\|u_0\|_{L^r(\Om)}}$} \Big)
	\cdot \frac{1}{\|u_0\|_{L^q(\Om)}}
	\qquad \mbox{for all } t>0.
  \ee
\end{lem}
\proof
  For integers $k\ge 0$, we set
  \be{13.2}
	p_k:=2^k + (2^k-1) \cdot \frac{n(1-m)}{2},
	\qquad \mbox{and} \qquad
	\gamma_k:=2^k \cdot \frac{1}{1-m} + (2^k-1) \cdot \frac{n}{2},
  \ee
  and for $k\ge 1$ we define
  \be{13.22}
	q_k:=\frac{p_k}{p_k-m+1} \cdot \frac{4}{4-n(m-1)},
   \ee
  and observe that the sequences $(p_k)_{k\ge 0}$ and $(\gamma_k)_{k\ge 0}$ are increasing and divergent with
  \be{13.3}
	p_k = 2p_{k-1} - \frac{n(m-1)}{2}
	\quad \mbox{and} \quad
	\gamma_k = 2\gamma_{k-1} + \frac{n}{2}
	\qquad \mbox{for all } k\ge 1,
  \ee
  while
  \be{13.4}
	0 < q_1 \le q_k \le \lim_{k'\to\infty} q_{k'} = \frac{4}{4-n(m-1)}<\infty
	\qquad \mbox{for all } k\ge 1.
  \ee
  As
  \bas
	\frac{2}{2-m} = \frac{2p_0}{p_0-m+1} \le \frac{2p_k}{p_k-m+1} \le 2
	\qquad \mbox{for all } k\ge 1
  \eas
  and,
  by (\ref{13.22}), (\ref{13.3}) and the fact that $p_{k-1}\ge 1$ for all $k\ge 1$ due to (\ref{13.2}),
   \bas
  q_k=\frac{p_k}{p_k-m+1} \cdot \frac{4}{4-n(m-1)}
  \le \frac{2p_{k-1}}{p_k-m+1}
  \qquad \mbox{for all } k\ge 1
  \eas
  as well as, by (\ref{13.3}) again,
  \bas
	\frac{2p_{k-1}}{p_k-m+1} = \frac{2p_{k-1}}{2p_{k-1} - \frac{(n+2)(m-1)}{2}}
	\qquad \mbox{for all } k\ge 1
  \eas
  and hence	
  \bas
	0 < \frac{2}{2-\frac{(n+2)(m-1)}{2}} \le \frac{2p_{k-1}}{p_k-m+1} \le 1
	\qquad \mbox{for all } k\ge 1,
  \eas
  we may draw on a Gagliardo-Nirenberg inequality to fix $c_1=c_1(m)>0$ such that for all $k\ge 1$
  and any $\vp\in W^{1,2}(\Om)$,
  \bas
	\|\vp\|_{L^\frac{2p_k}{p_k-m+1}(\Om)}^{\frac{n+2}{n} \cdot \frac{2p_k}{p_k-m+1}}
	\le c_1 \|\na\vp\|_{L^2(\Om)}^{\frac{n+2}{n} \cdot \frac{2p_k}{p_k-m+1} \cdot \iota_k}
		\|\vp\|_{L^\frac{2p_{k-1}}{p_k-m+1}(\Om)}^{\frac{n+2}{n} \cdot \frac{2p_k}{p_k-m+1} \cdot (1-\iota_k)}
	+ c_1\|\vp\|_{L^{q_k}(\Om)}^{\frac{n+2}{n} \cdot \frac{2p_k}{p_k-m+1}},
  \eas
  where
  \bas
	\iota_k:=\frac{n(p_k-m+1)(p_k-p_{k-1})}{p_k\cdot [(2-n)p_{k-1} + n(p_k-m+1)]}
	\qquad \mbox{and} \qquad
	1-\iota_k=\frac{p_{k-1} \cdot [2p_k-n(m-1)]}{p_k\cdot [(2-n)p_{k-1} + n(p_k-m+1)]}.
  \eas
  Straightforward simplification using (\ref{13.3}) hence shows that for any choice of $u_0$ and $u$ with the presupposed properties,
  for all $t>0$ and $k\ge 1$
  we have
  \bea{13.5}
	\bigg\{ \io u^{-p_k}\bigg\}^\frac{n+2}{n}
	&=& \big\| u^{-\frac{p_k-m+1}{2}}\big\|_{L^\frac{2p_k}{p_k-m+1}(\Om)}^{\frac{n+2}{n} \cdot \frac{2p_k}{p_k-m+1)}} \nn\\
	&\le& c_1\big\|\na u^{-\frac{p_k-m+1}{2}} \big\|_{L^2(\Om)}^2
		\big\| u^{-\frac{p_k-m+1}{2}} \big\|_{L^\frac{2p_{k-1}}{p_k-m+1}(\Om)}^{\frac{4}{n} \cdot \frac{2p_{k-1}}{p_k-m+1)}}
	+ c_1 \big\| u^{-\frac{p_k-m+1}{2}} \big\|_{L^{q_k}(\Om)}^{\frac{n+2}{n} \cdot \frac{2p_k}{p_k-m+1}} \nn\\
	&=& c_1 \big\| \na u^{-\frac{p_k-m+1}{2}} \big\|_{L^2(\Om)}^2 \cdot
		\bigg\{ \io u^{-p_{k-1}} \bigg\}^\frac{4}{n}
	+ c_1 \cdot \Big\| \frac{1}{u}\Big\|_{L^{\frac{2p_k}{4-n(m-1)}}(\Om)}^\frac{(n+2)p_k}{n},
  \eea
  because
  \bas
	\frac{(p_k-m+1)q_k}{2} =  \frac{p_k-m+1}{2} \cdot \frac{p_k}{p_k-m+1} \cdot \frac{4}{4-n(m-1)}
    =\frac{2p_k}{4-n(m-1)}
     \qquad \mbox{for all } k\ge 1
  \eas
  by (\ref{13.22}).\abs
  To make appropriate use of this, for $k\ge 0$ we define
  \be{13.6}
	M_k:=\sup_{t\in (0,K^{1-m})} \bigg\{ t^{\gamma_k} \io u^{-p_k}(\cdot,t) \bigg\},
  \ee
  where $K:=\zeta^\star \|u_0\|_{L^q(\Om)}$ with
  $\zeta^\star$ as in (\ref{zs}),
  and note that according to (\ref{13.2}) and Lemma \ref{lem12}, with $\Gamma_3=\Gamma_3(m,1,q,r)$
  as introduced there we have
  \be{13.7}
	M_0
	= \sup_{t\in (0,K^{1-m})} \bigg\{ t^\frac{1}{1-m} \io \frac{1}{u(\cdot,t)} \bigg\}
	\le \Gamma_3 \zeta^\star
	+ \frac{\Gamma_3}{\|u_0\|_{L^q(\Om)}} \cdot \big(\zeta^\star \|u_0\|_{L^q(\Om)}\big)
	= 2\Gamma_3 \zeta^\star
  \ee
  For $k\ge 1$, observing that $\frac{p(p+1)}{(p-m+1)^2} >\frac{p^2}{(p-m+1)^2} \ge c_2\equiv c_2(m):=\frac{1}{(2-m)^2}$ for all $p\ge 1$
  we may combine (\ref{0}) with (\ref{13.5}) to see that
  \bea{13.8}
	\frac{d}{dt} \io u^{-p_k}
	&=& - \frac{p_k(p_k+1)}{(p_k-m+1)^2} \io \big| \na u^{-\frac{p_k-m+1}{2}} \big|^2 \nn\\
	&\le& -c_2 \io \big| \na u^{-\frac{p_k-m+1}{2}} \big|^2 \nn\\
	&\le& - \frac{c_2}{c_1} \cdot \bigg\{ \io u^{-p_{k-1}} \bigg\}^{-\frac{4}{n}} \cdot
		\bigg\{ \io u^{-p_k} \bigg\}^\frac{n+2}{n} \nn\\
	& & + c_2\cdot\bigg\{ \io u^{-p_{k-1}} \bigg\}^{-\frac{4}{n}} \cdot
		\Big\| \frac{1}{u}\Big\|_{L^{\frac{2p_k}{4-n(m-1)}}(\Om)}^\frac{(n+2)p_k}{n}
	\qquad \mbox{for all } t>0,
  \eea
  where using that
  \bas
	\frac{2p_k}{4-n(m-1)} \ge \frac{2p_1}{4-n(m-1)} = 1
  \eas
  and
  \bas
	\frac{2p_k}{4-n(m-1)} - p_{k-1}
	= \frac{4p_{k-1} -n(m-1)}{4-n(m-1)} - \frac{4p_{k-1}-n(m-1) p_{k-1}}{4-n(m-1)}
	= \frac{n(m-1) (p_{k-1}-1)}{4-n(m-1)}
	\le 0
  \eas
  by (\ref{13.2}) and (\ref{13.3}), we may invoke the H\"older inequality to see that for all $t>0$,
  \be{13.9}
	c_2\cdot\bigg\{ \io u^{-p_{k-1}} \bigg\}^{-\frac{4}{n}} \cdot
		\Big\| \frac{1}{u}\Big\|_{L^\frac{2p_k}{4-n(m-1)}(\Om)}^\frac{(n+2)p_k}{n}
	\le c_2 \cdot \bigg\{ \io u^{-p_{k-1}} \bigg\}^{-\frac{4}{n}} \cdot
		\Big\| \frac{1}{u} \Big\|_{L^{p_{k-1}}(\Om)}^{\frac{(n+2)p_k}{n} \cdot \theta_k} \cdot
		\Big\| \frac{1}{u} \Big\|_{L^1(\Om)}^{\frac{(n+2)p_k}{n} \cdot (1-\theta_k)}
  \ee
  with
  \bas
	\theta_k:=\frac{p_{k-1}}{p_{k-1}-1} \cdot \Big( 1 - \frac{4-n(m-1)}{2p_k}\Big)
	\qquad \mbox{if } k\ge 2,
  \eas
  and with the actually arbitrary number $\theta_1$ for convenience be fixed according to $\theta_1:=\frac{4}{4-n(m-1)}$.
  Since thus from (\ref{13.3}) it follows that
  \bas
	\theta_k
	= \frac{p_{k-1}}{p_{k-1}-1} \cdot \frac{2p_k-4+n(m-1)}{2p_k}
	= \frac{p_{k-1}}{p_{k-1}-1} \cdot \frac{[4p_{k-1}-n(m-1)] - 4+n(m-1)}{2p_k}
	= \frac{2p_{k-1}}{p_k}
  \eas
  for all $k\ge 2$,
  and since in view of (\ref{13.2}) our choice of $\theta_1$ ensures that also $\theta_1=\frac{2p_0}{p_1}$,
  from (\ref{13.9}) we therefore obtain that in both cases $k=1$ and $k\ge 2$,
  \bea{13.99}
	& & \hs{-20mm}
	c_2\cdot\bigg\{ \io u^{-p_{k-1}} \bigg\}^{-\frac{4}{n}} \cdot
		\Big\| \frac{1}{u}\Big\|_{L^\frac{2p_k}{4-n(m-1)}(\Om)}^\frac{(n+2)p_k}{n} \nn\\
	&\le& c_2\cdot\bigg\{ \io u^{-p_{k-1}} \bigg\}^{-\frac{4}{n}} \cdot
		\Big\| \frac{1}{u} \Big\|_{L^{p_{k-1}}(\Om)}^\frac{2(n+2)p_{k-1}}{n} \cdot
		\Big\| \frac{1}{u} \Big\|_{L^1(\Om)}^{-\frac{(n+2)(m-1)}{2}}
	\qquad \mbox{for all } t>0,
  \eea
  because again by (\ref{13.3}),
  \bas
	\frac{(n+2)p_k}{n} \cdot (1-\theta_k)
	= \frac{(n+2)p_k}{n} \cdot \Big(1-\frac{2p_{k-1}}{p_k}\Big)
	= \frac{n+2}{n} (p_k-p_{k-1})
	= - \frac{(n-2)(m-1)}{2}.
  \eas
  We now recall our definition from (\ref{13.6}) in estimating
  \bas
	\io u^{-p_{k-1}} \le t^{-\gamma_{k-1}} M_{k-1}
	\qquad \mbox{for all } t\in (0,K^{1-m})
  \eas
  and
  \bas
	\io \frac{1}{u} = \io u^{-p_0}
	\le t^{-\gamma_0} M_0
	= t^{-\frac{1}{1-m}} M_0
	\qquad \mbox{for all } t\in (0,K^{1-m}),
  \eas
  which when applied to (\ref{13.99}) and (\ref{13.8}) shows that the latter entails the inequality
  \bas
	\frac{d}{dt} \io u^{-p_k}
	+ \frac{c_2}{c_1} t^\frac{4\gamma_{k-1}}{n} M_{k-1}^{-\frac{4}{n}} \cdot \bigg\{ \io u^{-p_k} \bigg\}^\frac{n+2}{n}
	&\le& c_2 \cdot (t^{-\gamma_{k-1}} M_{k-1})^2 \cdot (t^{-\frac{1}{1-m}} M_0)^{-\frac{(n+2)(m-1)}{2}} \\
	&=& c_2 t^{-2\gamma_{k-1}-\frac{n+2}{2}} M_0^{-\frac{(n+2)(m-1)}{2}} M_{k-1}^2
  \eas
  for all $t\in (0,K^{1-m})$.
  Employing Lemma \ref{lem76} with $a:=\frac{c_2}{c_1} M_{k-1}^{-\frac{4}{n}}, b:=c_2 M_0^{-\frac{(n+2)(m-1)}{2}}  M_{k-1}^2$,
  $\al:=\frac{n+2}{n}$, $\lam:=\frac{4\gamma_{k-1}}{n}$ and $\mu:=2\gamma_{k-1}+ \frac{n+2}{2}$, and noting that then,
  by (\ref{13.3}),
  \bas
	\frac{\lam+1}{\al-1}
	= \frac{\frac{4\gamma_{k-1}}{n}+1}{\frac{2}{n}}
	= 2\gamma_{k-1} + \frac{n}{2}
	= \gamma_k
  \eas
  as well as
  \bas
	\frac{\al+\lam}{\al-1} - \mu
	= \frac{n+2}{2} + \frac{\frac{4\gamma_{k-1}}{n}}{\frac{2}{n}}
	= \frac{n+2}{2} + 2\gamma_{k-1} - \mu
	= 0
  \eas
  turns this into the estimate
  \bas
	\io u^{-p_k} \le A_k t^{-\gamma_k}
	\qquad \mbox{for all } t\in (0,K^{1-m}),
  \eas
  where
  \bas
	A_k
	&:=&
	\max \Bigg\{ \bigg( \frac{2\gamma_k}{\frac{c_2}{c_1} M_{k-1}^{-\frac{4}{n}}} \bigg)^\frac{n}{2} \, , \,
		\bigg( \frac{2c_2 M_0^{-\frac{(n+2)(m-1)}{2}} M_{k-1}^2}{\frac{c_2}{c_1} M_{k-1}^{-\frac{4}{n}}} \bigg)^\frac{n}{n+2}
	\Bigg\} \\
	&=& \max \bigg\{ \Big(\frac{2c_1}{c_2}\Big)^\frac{n}{2} \gamma_k^\frac{n}{2} \, , \,
		(2c_1)^\frac{n}{n+2} M_0^{-\frac{n(m-1)}{2}} \bigg\} \cdot M_{k-1}^2 \\
	&\le& c_4 \cdot \big( 2^\frac{nk}{2} + M_0^{-\frac{n(m-1)}{2}} \big)\cdot M_{k-1}^2
  \eas
  with $c_4\equiv c_4(m):=(\frac{2c_1}{c_2})^\frac{n}{2} \cdot (\frac{1}{1-m} + \frac{n}{2})^\frac{n}{2} + (2c_1)^\frac{n}{n+2}$,
  because from (\ref{13.2}) we know that $\gamma_k\le (\frac{1}{1-m}+\frac{n}{2}) \cdot 2^k$.
  In line with (\ref{13.6}), we therefore obtain that
  \bas
	M_k \le A_k \le c_4\cdot\big(2^\frac{nk}{2} + M_0^{-\frac{n(m-1)}{2}}\big) \cdot M_{k-1}^2
	\qquad \mbox{for all } k\ge 1,
  \eas
  and that hence, by Lemma \ref{lem79},
  \bas
	\ln M_k
	&\le& \ln c_4 + \ln \big(2^\frac{nk}{2} - M_0^{-\frac{n(m-1)}{2}}\big) + 2\ln M_{k-1} \\
	&\le& \ln c_4 + \frac{n\ln 2}{2} \cdot k
	- \frac{n(m-1)}{2} \cdot \ln_+ M_0 + \ln 2 + 2\ln M_{k-1}
	\qquad \mbox{for all } k\ge 1.
  \eas
  Lemma \ref{lem78} thus ensures that
  \bas
	\ln M_k
	&\le& 2^k \ln M_0
	+ \sum_{j=1}^k 2^{k-j} \cdot \Big\{ \ln (2c_4) + \frac{n\ln 2}{2} \cdot j - \frac{n(m-1)}{2} \cdot \ln_+ M_0\Big\} \\
	&=& 2^k \ln M_0 + \Big\{ \ln (2c_4) - \frac{n(m-1)}{2} \cdot \ln_+ M_0 \Big\} \cdot \sum_{j=1}^k 2^{k-j}
	+ \frac{n\ln 2}{2} \cdot \sum_{j=1}^k j\cdot 2^{k-j}
	\qquad \mbox{for all } k\ge 1.
  \eas
  Since
  $\sum_{j=1}^k 2^{k-j} = 2^{k}-1 <2^{k+1}$
  and
  $\sum_{j=1}^k j\cdot 2^{k-j} = 2^{k+1} - k - 2 \le 2^{k+1}$
  for all $k\ge 1$,
  on letting $c_5\equiv c_5(m):=\max \{\ln (2c_4) \, , \, \frac{n\ln 2}{2} \, , \, 1\} \ge 1$ we thereby obtain that
  \bas
	\ln M_k
	&\le& 2^k \ln M_0	
	+ 2^{k+1} c_5
	- 2^k \cdot n(m-1) \ln_+ M_0
	+ 2^{k+1} c_5 \\
	&\le& 2^k \cdot \big\{ 1+n(1-m) \big\} \cdot \ln_+ M_0 + 2^{k+2} c_5
	\qquad \mbox{for all } k\ge 1.
  \eas
  As $p_k\ge 2^k$ for all $k\ge 1$ by (\ref{13.2}), this entails that
  \bas
	\ln M_k^\frac{1}{p_k}
	&\le& \frac{2^k}{p_k} \cdot \big\{ 1+n(1-m) \big\} \cdot \ln_+ M_0
	+ \frac{2^{k+2} c_5}{p_k} \\
	&\le& \big\{1+n(1-m)\big\} \cdot \ln_+ M_0
	+ 4c_5
	\qquad \mbox{for all } k\ge 1,
  \eas
  in view of (\ref{13.6}) and (\ref{13.7}) meaning that with $c_6\equiv c_6(m,q,r):=\max\{1,2\Gamma_3\}$ and
  $c_7\equiv c_7(m,q,r):=\exp\big\{ 4c_5 + \{1+n(1-m)\} \cdot \ln c_6\big\}$,
  \bas
	\Big\| \frac{1}{u}\Big\|_{L^{p_k}(\Om)}
	&\le& e^{4c_5} \cdot e^{\{1+n(1-m)\} \cdot \ln_+ M_0} \cdot t^{-\frac{\gamma_k}{p_k}} \\
	&\le& e^{4c_5} \cdot \exp \bigg\{ \big\{ (1+n(1-m)\big\} \cdot
		\ln ( c_6 \zeta^\star  + c_6 ) \bigg\}
		\cdot t^{-\frac{\gamma_k}{p_k}} \\
	&=& c_7 \cdot ( \zeta^\star + 1 )^{1+n(1-m)}
		\cdot t^{-\frac{\gamma_k}{p_k}}
	\qquad \mbox{for all $t\in (0,K^{1-m}]$ and $k\ge 1$.}
  \eas
  Since $\frac{\gamma_k}{p_k} \to \frac{1}{1-m}$ as $k\to\infty$, we thus conclude that
  \bas
	\Big\| \frac{1}{u}\Big\|_{L^\infty(\Om)}
	\le c_7 \cdot ( \zeta^\star + 1 )^{1+n(1-m)}
		\cdot t^{-\frac{1}{1-m}}
	\qquad \mbox{for all } t\in (0,K^{1-m}],
  \eas
  from which it follows that if we let $c_8\equiv c_8(q,r):=|\Om|^\frac{r-q}{qr}$ and
  $c_9\equiv c_9(m,q,r):= c_7 \cdot \Big(\frac{1}{\zeta_{\frac{qr}{r-q}}(c_8)} +1\Big)^{1+n(1-m)}$, then in line with (\ref{zs}),
  \bas
	\Big\| \frac{1}{u}\Big\|_{L^\infty(\Om)}
	\le c_9 (\zeta^\star)^{1+n(1-m)}
		\cdot t^{-\frac{1}{1-m}}
	\qquad \mbox{for all } t\in (0,K^{1-m}],
  \eas
  because $\frac{\|u_0\|_{L^q(\Om)}}{\|u_0\|_{L^r(\Om)}} \le c_8$ by the H\"older inequality, and because by downward monotonicity
  of $\zeta_{\frac{qr}{r-q}}$ we have $(1+\zeta_{\frac{qr}{r-q}}(\xi))^{1+n(1-m)} \le c_9 \zeta_{\frac{qr}{r-q}}^{1-n(1-m)}(\xi)$
  for all $\xi\in (0,c_8]$.
  Since $0<1\mapsto \|\frac{1}{u(\cdot,t)}\|_{L^\infty(\Om)}$ is nonincreasing by the maximum principle, this particularly implies that
  \bas
	\Big\| \frac{1}{u}\Big\|_{L^\infty(\Om)}
	\le c_9 (\zeta^\star)^{1+n(1-m)}
		\cdot \frac{1}{K}
	= c_9 (\zeta^\star)^{n(1-m)}
		\cdot \frac{1}{\|u_0\|_{L^q(\Om)}}
	\qquad \mbox{for all } t> K^{1-m},
  \eas
  whence again recalling (\ref{zs}) we see that the inequality in (\ref{13.1}) holds for all $t>0$ if we let
  $\Gamma_4\equiv \Gamma_4(m,q,r):=c_9$.
\qed
\mysection{Approximation by smooth solutions. Proof of the main results}\label{sect4}
The setting of general and possibly non-smooth initial data as in Proposition \ref{prop66} and Theorems \ref{theo67} and \ref{theo68}
can now be reduced in a fairly straightforward manner by relying on a standard $L^1$-contraction property associated with (\ref{0}).
The following actually more general statement in this regard translates meanwhile traditional knowledge to the present scenario
of (\ref{0}) (cf., e.g., \cite{ACP} and \cite{V1}). The comparison principle additionally contained will be referred
to in the proofs of Propositions \ref{prop56} and \ref{prop55} below.
\begin{lem}\label{lem61}
  Let $n\ge 1$ and $m\in (-\infty,1)$, and suppose that for some $t_0\in\R$ and $T\in (t_0,\infty]$, the functions
  $\uu$ and $\ou$ lie in $C^0([t_0,T);L^1(\Om)) \cap C^{2,1}(\bom\times (t_0,T))$ and are such that
  $\uu>0$ and $\ou>0$ in $\bom\times (t_0,T)$, that
  \be{61.1}
	\uu_t \le \na\cdot (\uu^{m-1}\na\uu)
	\quad \mbox{and} \quad
	\ou_t \ge \na\cdot (\ou^{m-1}\na\ou)
	\qquad \mbox{in } \Om\times (t_0,T),
  \ee
  and that
  \be{61.2}
	\frac{\pa \ou}{\pa\nu} \le 0 \le \frac{\pa \ou}{\pa\nu}
	\qquad \mbox{on } \pO\times (t_0,T).
  \ee
  Then
  \be{61.22}
	\io \big(\uu(\cdot,t)-\ou(\cdot,t)\big)_+
	\le
	\io \big(\uu(\cdot,t_0)-\ou(\cdot,t_0)\big)_+
	\qquad \mbox{for all } t\in (t_0,T).
  \ee
\end{lem}
\proof
  We let $(\vp_\del)_{\del\in (0,1)} \subset C^1(\R)$ be such that for each $\del\in (0,1)$, the function $\vp_\del$
  is nondecreasing and such that $\vp_\del\equiv 0$ on $(-\infty,0]$ and $\vp_\del\equiv 1$ on $[\del,\infty)$, and we define
  $\Phi_\del(\xi):=\int_0^\xi \vp_\del(\sig)d\sig$ for $\xi\in\R$ and $\del\in (0,1)$.
  Then for any $\del\in (0,1)$,
  \be{61.3}
	0 \le \vp_\del(\xi)
	\le \sgn_+ \xi:=\lball
		0 \quad & \mbox{if } \xi\le 0, \\[1mm]
		1 \quad & \mbox{if } \xi>0
	\ear
	\qquad \mbox{and} \qquad
	0 \le\Phi_\del(\xi) \le \xi_+
	\quad \mbox{for all } \xi\in\R,
  \ee
  and for each $\xi\in\R$ we have
  \be{61.4}
	\vp_\del(\xi) \to \sgn_+ \xi
	\quad \mbox{and} \quad
	\Phi_\del(\xi) \to \xi_+
	\qquad \mbox{as } \del\searrow 0.
  \ee
  Now letting $\psi(\xi):=\int_1^\xi \sig^{m-1}$, $\xi>0$, from (\ref{61.1}) we obtain that in the classical sense, $d:=\uu-\ou$
  satisfies
  $d_t \le \Del (\psi(\uu)-\psi(\ou))$ in $\Om\times (t_0,T)$, which upon multiplying by $\vp_\del\big( \psi(\uu)-\psi(\ou)\big)$
  and integrating by parts shows that for any fixed $t_1\in (t_0,T)$ and $t\in (t_1,T)$,
  \bas
	\int_{t_1}^t \io d_t \cdot \vp_\del\big(\psi(\uu)-\psi(\ou)\big)
	&\le& - \int_{t_1}^t \io \vp_\del'\big( \psi(\uu)-\psi(\ou)\big) \big| \na \big(\psi(\uu)-\psi(\ou)\big)\big|^2 \\
	& & + \int_{t_1}^t \int_{\pO} \vp_\del \big( \psi(\uu)-\psi(\ou)\big) \frac{\pa \big(\psi(\uu)-\psi(\ou)\big)}{\pa\nu} \\
	&\le& 0,
  \eas
  because
  \bas
	\frac{\pa \big(\psi(\uu)-\psi(\ou)\big)}{\pa\nu}
	= \uu^{m-1} \frac{\pa \uu}{\pa\nu} - \ou^{m-1} \frac{\pa \ou}{\pa\nu}
	\le 0
	\qquad \mbox{on } \pO\times (t_0,T)
  \eas
  by (\ref{61.2}).
  Using that our hypotheses warrant that $d_t$ is continuous on $\bom\times [t_1,t]$, in view of (\ref{61.3}) and (\ref{61.4})
  we may employ the dominated convergence theorem to thus find that since $\psi(\uu)-\psi(\ou)$ is positive precisely at those
  points at which $\uu-\ou>0$,
  \bas
	0
	\ge
	\int_{t_1}^t \io d_t \cdot \vp_\del \big(\psi(\uu)-\psi(\ou)\big)
	\to \int_{t_1}^t \io d_t \sgn_+ \big(\psi(\uu)-\psi(\ou)\big)
	= \int_{t_1}^t \io d_t \sgn_+ (\uu-\ou)
	\qquad \mbox{as } \del\searrow 0.
  \eas
  Since, by the same token, also
  \bas
	\int_{t_1}^t \io d_t \vp_\del(d)
	= \int_{t_1}^t \io d_t \vp_\del(\uu-\ou)
	\to \int_{t_1}^t \io d_t \sgn_+ (\uu-\ou)
	\qquad \mbox{as } \del\searrow 0,
  \eas
  we thereby obtain that
  \bas
	0 &\ge&
	\limsup_{\del\searrow 0} \int_{t_1}^t \io d_t \vp_\del(d)
	= \limsup_{\del\searrow 0}\int_{t_1}^t \frac{d}{dt'} \io \Phi_\del(d(x,t')) dx dt' \\
	&=& \limsup_{\del\searrow 0} \bigg\{ \io \Phi_\del(d(x,t)) dx
	- \io \Phi_\del(d(x,t_1)) dx \bigg\}.
  \eas
  Since, again by the dominated convergence theorem, the respective second parts in (\ref{61.3}) and (\ref{61.4}) along with the
  continuity of $d_+$ in $\bom\times (t_0,T)$ ensure that for each $t'\in (t_0,T)$ we have
  \bas
	\io \Phi_\del(d(x,t')) dx \to \io d_+(x,t') dx
	\qquad \mbox{as } \del\searrow 0,
  \eas
  we thus conclude that
  \be{61.5}
	\io d_+(\cdot,t) \le \io d_+(\cdot,t_1)
	\qquad \mbox{for all $t_1\in (t_0,T)$ and } t\in (t_1,T).
  \ee
  It remains to note that $d(\cdot,t_1)\to d(\cdot,t_0)$ in $L^1(\Om)$ as $t_1\searrow t_0$ according to our hypotheses,
  and that hence
  \bas
	\io d_+(\cdot,t_1)\to \io d_+(\cdot,t_0)
	\qquad \mbox{as } t_1\searrow t_0
  \eas
  thanks to the dominated convergence theorem, to infer (\ref{61.22}) from (\ref{61.5}).
\qed
Having this at hand, we can first derive the qualitative statement on existence, uniqueness and approximation of locally smooth
$L^1$-continuous solutions by yet quite coarse applications of Lemma \ref{lem3} and Lemma \ref{lem13}:\abs
\proofc of Proposition \ref{prop66}. \quad
  Given any $(u_{0j})_{j\in\N} \subset C^\infty(\bom)$ such that $u_{0j}>0$ in $\bom$ for all $j\in\N$ and
  $u_{0j} \to u_0$ in $L^s(\Om)$ as $j\to\infty$, two applications of Lemma \ref{lem61} show that the positive classical solutions
  $u_j\in C^0(\bom\times [0,\infty)) \cap C^{2,1}(\bom\times (0,\infty))$ emanating from $u_j$, for each $j\in\N$ existing
  thanks to standard parabolic solution theory and the maximum principle (\cite{amann}),
  satisfy $\|u_j(\cdot,t)-u_{j'}(\cdot,t)\|_{L^1(\Om)} \le \|u_{0j}-u_{0j'}\|_{L^1(\Om)}$ for all $t\ge 0$, $j\in\N$ and $j'\in\N$.
  Since $s\ge 1$ and thus $u_{0j} \to u_0$ in $L^1(\Om)$ as $j\to\infty$, this ensures that for every $T>0$,
  $(u_j)_{j\in\N}$ forms a Cauchy sequence in $C^0([0,T];L^1(\Om))$, and that thus
  \be{66.2}
	u_j\to u
	\quad \mbox{in } C^0_{loc}([0,\infty);L^1(\Om))
	\qquad \mbox{as } j\to\infty
  \ee
  with some nonnegative $u\in C^0([0,\infty);L^1(\Om))$ fulfilling $u(\cdot,0)=u_0$.\abs
  Next, relying on (\ref{66.1}) and the mere fact that $\sup_{j\in\N} \|u_{0j}\|_{L^s(\Om)} <\infty$,
  we may employ Lemma \ref{lem3} to see that for each $\tau>0$ there exists $c_1=c_1(\tau)>0$ such that
  \be{66.3}
	u_j \le c_1
	\quad \mbox{in $\Om\times (\frac{\tau}{2},\infty)$ \qquad for all } j\in\N,
  \ee
  whereas, similarly, letting $r:=s\ge 1$ and choosing any $q\in (0,1)$, and using that then
  $\sup_{j\in\N} \|u_{0j}\|_{L^r(\Om)} <\infty$ and $\inf_{j\ge j_0} \|u_{0j}\|_{L^q(\Om)} >0$ for some $j_0\in\N$
  due to the fact that $u_0\not\equiv 0$, for any $\tau>0$ we can employ Lemma \ref{lem13} find $c_2(\tau)>0$ satisfying
  \be{66.4}
	u_j\ge c_2
	\quad \mbox{in $\Om\times (\frac{\tau}{2},\infty)$ \qquad for all } j\ge j_0.
  \ee
  Thanks to the non-degeneracy property implied by (\ref{66.3}) and (\ref{66.4}), standard parabolic Schauder theory (\cite{LSU})
  applies so as to assert that for each $\tau>0$ and $T>\tau$ there exists $\theta=\theta(\tau,T)\in (0,1)$ such that
  $(u_j)_{j\ge j_0}$ is bounded in $C^{2+\theta,1+\frac{\theta}{2}}(\bom\times [\tau,T])$, in view of the Arzel\`a-Ascoli theorem
  meaning that $(u_j)_{j\ge j_0}$ is relatively compact in $C^{2,1}_{loc}(\bom\times (0,\infty))$.
  The convergence in (\ref{66.2}) therefore additionally takes place in $C^{2,1}_{loc}(\bom\times (0,\infty))$,
  whence taking limits in (\ref{0}) we obtain that indeed $u$ satisfies $u_t=\na\cdot (u^{m-1}\na u)$ in $\Om\times (0,\infty)$
  and $\frac{\pa u}{\pa\nu}=0$ on $\pO\times (0,\infty)$.
\qed
Relying on the part of Proposition \ref{prop66} that is concerned with approximation, the $L^p$-estimates
in Theorem \ref{theo67} can readily be obtained by now making full use of Lemma \ref{lem3}, and of Lemma \ref{lem2}.\abs
\proofc of Theorem \ref{theo67}. \quad
  We take any $(u_{0j})_{j\in\N} \subset C^\infty(\bom)$ such that $u_{0j}>0$ in $\bom$ and
  $\|u_{0j}\|_{L^s(\Om)} \le 2\|u_0\|_{L^s(\Om)}$ for all $j\in\N$, and that $u_{0j} \to u_0$ in $L^s(\Om)$ as $j\to\infty$. Then
  Lemma \ref{lem1} and Lemma \ref{lem3} assert that with $\Gamma_1$ and $\Gamma_2$ as introduced in Lemma \ref{lem1}
  and Lemma \ref{lem2}, and with
  \bas
	\Gamma^\star(m,s,p):=\lball
	\Gamma_1(m,s,p)
	\qquad & \mbox{if } p<\infty, \\[1mm]
	\Gamma_2(m,s)
	\qquad & \mbox{if } p=\infty,
	\ear
  \eas
  as found there, the corresponding solutions $u_j$ of (\ref{0}) emanating from $u_{0j}$ satisfy
  \bas
	\|u_j(\cdot,t)\|_{L^p(\Om)}
	\le \Gamma^\star \|u_0\|_{L^s(\Om)}^\frac{2s+\frac{n(m-1)s}{p}}{2s+n(m-1)} \cdot t^{-\frac{n(1-\frac{s}{p})}{2s+n(m-1)}}
		+ \Gamma^\star \|u_0\|_{L^s(\Om)}
	\qquad \mbox{for all $t>0$ and } j\in\N.
  \eas
  In view of the approximation statement in Proposition \ref{prop66}, this entails (\ref{67.2}).
\qed
The claimed lower bounds can be verified quite similarly:\abs
\proofc of Theorem \ref{theo68}. \quad
  In much the same manner as Theorem \ref{theo67} has been derived from Proposition \ref{prop66} as well as
  Lemma \ref{lem1} and Lemma \ref{lem3}, this can be seen by combining Proposition \ref{prop66} with Lemma \ref{lem13}.
\qed
Without substantial further efforts, the above theory can be seen to actually transfer one-to-one to the solution
framework specified in Definition \ref{dw}:\abs
\proofc of Corollary \ref{cor57}. \quad
  We let $u$ be as in Proposition \ref{prop66}, and then directly obtain from the inclusion $u\in C^{2,1}(\bom\times (0,\infty))$
  that for $\Psi$ as in (\ref{Psi}) we have $\Psi(u)\in L^\infty_{loc}(\bom\times (0,\infty))$.
  As the property in (\ref{0w}) is immediately implied by the fact that in the classical sense we have
  $u_t=\Del \Psi(u)$ in $\bom\times (0,\infty)$ and $\frac{\pa u}{\pa\nu}=0$ on $\pO\times (0,\infty)$ according to
  Definition \ref{ds}, it follows that $u$ indeed also is a weak solution of (\ref{0}).\abs
  To see its uniqueness, we suppose that $\wh{u}$ is another weak solution of (\ref{0}), and then obtain by an almost verbatim
  copy of a well-established argument, as detailed in \cite[Proposition 9]{ACP},
  for instance, that for any $T>0$ and $t_0\in (0,T)$
  we have
  \bas
	\io (u-\wh{u})_+ (\cdot,T)
	\le
	\io (u-\wh{u})_+ (\cdot,t_0)
	\qquad \mbox{and} \qquad
	\io (\wh{u}-u)_+ (\cdot,T)
	\le
	\io (\wh{u}-u)_+ (\cdot,t_0)
  \eas
  and hence
  \bas
	\|u(\cdot,T)-\wh{u}(\cdot,T)\|_{L^1(\Om)} \le
	\|u(\cdot,t_0)-\wh{u}(\cdot,t_0)\|_{L^1(\Om)}.
  \eas
  Since both $u(\cdot,t_0)\to u_0$ and $\wh{u}(\cdot,t_0) \to u_0$ in $L^1(\Om)$ as $t_0\searrow 0$,
  and since $T>0$ was arbitrary here, this confirms that, in fact, $u=\wh{u}$ a.e.~in $\Om\times (0,\infty)$.
\qed
\mysection{Optimality. Proofs of Propositions \ref{prop56} and \ref{prop55}}\label{sect5}
The statements both of Proposition \ref{prop56} and of Proposition \ref{prop55} will be derived by applications of the comparison
principle implied by Lemma \ref{lem61} to sub- and supersolutions which fit into the template generated as follows.
\begin{lem}\label{lem555}
  Let $n\ge 1$ and $m\in (-\infty,1)$, and suppose that there exists $r_1>0$ such that $\ov{B}_{r_1}(0)\subset \Om$.
  Then whenever $\rho\in C^2([0,\infty))$ is such that $\rho\ge 0$ and $\rho'(r)=0$ for all $r\ge r_1$,
  for any choice of $\beta\in\R$ and $T>0$, and for arbitrary functions $y\in C^0([0,T)) \cap C^1((0,T))$ and
  $\del\in C^0([0,T)) \cap C^1((0,T))$ which are both positive on $(0,T)$, the function defined by
  \be{555.01}
	u(x,t):=y(t)\cdot \big\{ \rho^2(|x|) + \del(t)\big\}^\frac{\beta}{2},
	\qquad x\in\bom, \ t\in [0,T),
  \ee
  belongs to $C^0(\bom\times [0,T)) \cap C^{2,1}(\bom\times (0,T))$ and satisfies
  $\frac{\pa u}{\pa\nu}=0$ on $\pO\times (0,T)$ as well as
  \bea{555.3}
	u_t - \na\cdot (u^{m-1}\na u)
	&=& y'(t) \cdot \big\{ \rho^2(|x|) + \del(t) \big\}^\frac{\beta}{2}
	+ \frac{\beta}{2} y(t)\del'(t) \cdot \big\{ \rho^2(|x|) + \del(t) \big\}^\frac{\beta-2}{2} \nn\\
	& & - \beta(m\beta-2) y^m(t) \rho^2(|x|) \rho'^2(|x|) \cdot \big\{ \rho^2(|x|) + \del(t) \big\}^\frac{m\beta-4}{2} \nn\\
	& & - \beta y^m(t) \rho(|x|) \rho''(|x|) \cdot \big\{ \rho^2(|x|) + \del(t) \big\}^\frac{m\beta-2}{2} \nn\\
	& & - \beta y^m(t) \cdot \Big\{ \rho'^2(|x|) + \frac{n-1}{|x|} \rho(|x|)\rho'(|x|)\Big\} \cdot
		\big\{ \rho^2(|x|) + \del(t) \big\}^\frac{m\beta-2}{2}
  \eea
  for all $x\in\Om$ and $t\in (0,T)$.
\end{lem}
\proof
  Switching to radial notation when writing $r=|x|$ for $x\in\bom$, we compute
  \bas
	u_t = y'(t) \cdot \big\{ \rho^2(r) + \del(t) \big\}^\frac{\beta}{2}
	+ \frac{\beta}{2} y(t)\del'(t) \cdot \big\{ \rho^2(r) + \del(t) \big\}^\frac{\beta-2}{2}
  \eas
  and
  \be{555.4}
	u_r=\beta y(t) \rho(r) \rho'(r) \cdot \big\{ \rho^2(r) + \del(t) \big\}^\frac{\beta-2}{2}
  \ee
  as well as
  \bas
	u_{rr}
	&=&
	\beta(\beta-2) y(t) \rho^2(r) \rho'^2(r) \cdot \big\{ \rho^2(r) + \del(t) \big\}^\frac{\beta-4}{2}
	+ \beta y(t) \rho(r) \rho''(r) \cdot \big\{ \rho^2(r) + \del(t) \big\}^\frac{\beta-2}{2} \\
	& & + \beta y(t) \rho'^2(r) \cdot \big\{ \rho^2(r) + \del(t) \big\}^\frac{\beta-2}{2},
  \eas
  and hence obtain (\ref{555.3}) by straightforward rearrangement.
  As (\ref{555.4}) the fact that $\rho'\equiv 0$ on $[r_1,\infty)$ entail that
  $\supp \na u \subset \ov{B}_{r_1}(0) \times [0,T]$, the inclusion $\ov{B}_{r_1}(0)\subset\Om$ ensures that
  $\na u\equiv 0$ in a neighborhood of $\pO\times (0,T)$, and that thus, in particular,
  $\frac{\pa u}{\pa\nu}=0$ on $\pO\times (0,T)$.
\qed
By choosing $\delta=const.$ and $\rho$ to be a cut-off variant of the identity here, within suitably short time intervals
we can find a positive decreasing $y$ such that the separated function then defined by (\ref{555.01}) acts as a subsolution for
(\ref{0}). In the limit $\del\searrow 0$, this shows that corresponding initial data cannot lead to solutions which
are smooth for $t>0$:\abs
\proofc of Proposition \ref{prop56}.\quad
  Without loss of generality assuming that $0\in\Om$, we fix $r_1>0$ and $R>0$ such that $\ov{B}_{r_1}(0)\subset\Om \subset B_R(0)$,
  and taking an $r_0\in (0,r_1)$ we let $\rho\in C^\infty([0,\infty))$ be such that
  \be{555.1}
	\rho(r)=r
	\quad \mbox{for all } r\in [0,r_0],
	\qquad
	\rho''(r)\le 0
	\quad \mbox{for all } r\ge 0
	\qquad \mbox{and} \qquad
	\rho'(r)=0
	\quad \mbox{for all } r\ge r_1,
  \ee
  noting that then, since $\rho(r)\ge \rho(r_0)=r_0>0$ for all $r>r_0$, it follows that there exists $c_1>0$ such that
  \be{56.2}
	\rho(r)\ge c_1 r
	\qquad \mbox{for all } r\in [0,R].
  \ee
  Assuming that $u_0\in L^1(\Om)$ satisfies (\ref{56.1}), taking $c_2>0$ such that
  \be{56.3}
	u_0(x) \ge c_2 |x|^{-\frac{2}{1-m}}
	\qquad \mbox{for a.e.~} x\in\Om
  \ee
  we then let $y_0:=c_1^\frac{2}{1-m} c_2$ and observe that for each $\del>0$,
  \bas
	\uu(x,t)\equiv \uu^{(\del)}(x,t)
	:=y(t) \cdot \big\{ \rho^2(|x|)+\del\big\}^{-\frac{1}{1-m}},
	\qquad x\in\bom, \ t\in [0,T),
  \eas
  with $T:=\frac{y_0^{1-m}}{2n}$ and
  \be{56.33}
	y(t):=(y_0^{1-m} -2nt)^\frac{1}{1-m},
	\qquad t\in [0,T),
  \ee
  satisfies
  \be{56.4}
	\uu(x,0)
	= y_0 \cdot \big\{ \rho^2(|x|)+\del\big\}^{-\frac{1}{1-m}}
	\le y_0 \cdot \rho^{-\frac{2}{1-m}}(|x|)
	\le c_1^{-\frac{2}{1-m}} y_0 |x|^{-\frac{2}{1-m}}
	= c_2 |x|^{-\frac{2}{1-m}}
	\le u_0(x)
  \ee
  for a.e.~$x\in\Om$ according to (\ref{56.2}) and (\ref{56.3}).
  An application of Lemma \ref{lem555} shows that furthermore
  \be{56.5}
	\frac{\pa \uu}{\pa\nu}=0
	\qquad \mbox{on } \pO\times (0,T),
  \ee
  and that
  \bas
	\uu_t - \na\cdot (\uu^{m-1} \na\uu)
	&=& y'(t) \cdot \big\{ \rho^2(|x|)+\del\big\}^{-\frac{1}{1-m}} \\
	& & - \frac{4}{(1-m)^2} y^m(t) \rho^2(|x|) \rho'^2(|x|) \cdot \big\{ \rho^2(|x|)+\del\big\}^{-\frac{1}{m}-1} \\
	& & + \frac{2}{1-m} y^m(t) \rho(|x|) \rho''(|x|) \cdot \big\{ \rho^2(|x|)+\del\big\}^{-\frac{1}{1-m}} \\
	& & + \frac{2}{1-m} y^m(t) \cdot \Big\{ \rho'^2(|x|) + \frac{n-1}{|x|} \rho(|x|) \rho'(|x|)\Big\} \cdot
		\big\{ \rho^2(|x|)+\del\big\}^{-\frac{1}{1-m}} \\
	&\le& y'(t) \cdot \big\{ \rho^2(|x|)+\del\big\}^{-\frac{1}{1-m}} \\
	& & + \frac{2}{1-m} y^m(t) \cdot \Big\{ \rho'^2(|x|) + \frac{n-1}{|x|} \rho(|x|) \rho'(|x|) \Big\} \cdot
		\big\{ \rho^2(|x|)+\del\big\}^{-\frac{1}{1-m}}
  \eas
  for all $x\in\Om$ and $t\in (0,T)$,
  because $\rho\ge 0$ and $\rho''\le 0$ by (\ref{555.1}). Since necessarily $0\le \rho'\le 1$ and hence
  \bas
	\rho'^2(|x|) + \frac{n-1}{|x|} \rho(|x|)\rho(|x|)
	\le 1 + \frac{n-1}{|x|} \cdot |x| \cdot 1=n
	\qquad \mbox{for all } x\in\Om,
  \eas
  and since (\ref{56.33}) ensures that
  \bas
	y'(t)=-\frac{2n}{1-m} y^m(t)
	\qquad \mbox{for all } t\in (0,T),
  \eas
  this implies that
  \be{56.6}
	\uu_t - \na\cdot (\uu^{m-1} \na\uu) \le 0
	\qquad \mbox{in } \Om\times (0,T).
  \ee
  Now if (\ref{0}) had a locally smooth $L^1$-continuous solution $u$, then on the basis of (\ref{56.4}), (\ref{56.6}) and (\ref{56.5})
  the comparison principle from Lemma \ref{lem61} would assert that whenever $\del>0$, we have
  \bas
	u(x,t) \ge \uu^{(\del)}(x,t)
	\qquad \mbox{for all $x\in\Om$ and } t\in (0,T).
  \eas
  On taking $\del\searrow 0$ and using that $\rho(0)=0$ by (\ref{555.1}), from this we would infer that $u$ could not be smooth
  in $\Om\times (0,T)$.
\qed
For $\del\sim id$ and $y\equiv 1$, suitably small power-type behavior of $\rho$ lets the function in (\ref{555.01})
become a supersolution which can be used to show that the lower estimate in Theorem \ref{theo68} cannot significantly
be improved:\abs
\proofc of Proposition \ref{prop55}.\quad
  Again assuming that $0\in\Om$, we pick $r_0>0$ and $r_1>r_0$ such that $\ov{B}_{r_1}(0)\subset\Om$, and given $\eps\in (0,1)$
  we fix $\rho\in C^\infty([0,\infty))$ such that
  \be{55.2}
	\rho(r)=r^\frac{1}{\eps}
	\quad \mbox{for all } r\in [0,r_0],
	\qquad
	\rho'(r)\ge 0
	\quad \mbox{for all } r\ge 0
	\qquad \mbox{and} \qquad
	\rho'(r)=0
	\quad \mbox{for all } r\ge r_1.
  \ee
  Writing $c_1:=\rho(r_0), c_2:=\rho(r_1), c_3:=\sup_{r\ge r_0} \rho'(r)$ and $c_4:=\sup_{r\ge r_0} |\rho''(r)|$, we can then choose
  $a>0$ suitably large such that
  \be{55.3}
	a \ge \Big( \frac{4|m|}{1-m}+1\Big)\cdot\frac{1}{\eps^2} + \frac{1-\eps}{\eps^2} + \frac{n-1}{\eps}
  \ee
  and
  \be{55.4}
	a \ge \Big( \frac{4|m|}{1-m}+1\Big)\cdot c_1^{-2+2\eps} c_3^2
	+ \Big( c_4 + \frac{(n-1)c_3}{r_0}\Big) \cdot \max \big\{ c_1^{-1+2\eps} \, , \, c_2^{-1+2\eps} \big\},
  \ee
  and we thereupon define
  \bas
	\ou(x,t):=\big\{ \rho^2(|x|) + at \big\}^\frac{1-\eps}{1-m},
	\qquad x\in\bom, \ t\ge 0.
  \eas
  Then thanks to (\ref{55.2}), we may employ Lemma \ref{lem555} to see that $\frac{\pa\ou}{\pa\nu}=0$ on $\pO\times (0,\infty)$, and
  that
  \bea{55.6}
	\hs{-10mm}
	\ou_t- \na \cdot (\ou^{m-1}\na\ou)
	&=& \frac{1-\eps}{1-m} \cdot a \cdot \big\{ \rho^2(|x|) + at \big\}^\frac{m-\eps}{1-m} \nn\\
	& & - \frac{1-\eps}{1-m} \cdot \frac{4m-2-2\eps}{1-m} \cdot \rho^2(|x|) \rho'^2(|x|) \cdot
		\big\{ \rho^2(|x|) + at \big\}^\frac{(3-\eps)m-2}{1-m} \nn\\
	& & - \frac{1-\eps}{1-m} \cdot \rho(|x|) \rho''(|x|) \cdot \big\{ \rho^2(|x|) + at \big\}^\frac{(2-\eps)m-1}{1-m} \nn\\
	& & - \frac{1-\eps}{1-m} \cdot \Big\{ \rho'^2(|x|) + \frac{n-1}{|x|} \rho(|x|) \rho'(|x|) \Big\} \cdot
		\big\{ \rho^2(|x|) + at \big\}^\frac{(2-\eps)m-1}{1-m} \nn\\
	&=& \frac{1-\eps}{1-m} \cdot \big\{ \rho^2(|x|) + at \big\}^\frac{m-\eps}{1-m} \cdot \big\{ a - I(x,t) \big\}
	\qquad \mbox{for all $x\in\Om$ and } t>0,
  \eea
  where
  \bas
	I(x,t)
	&:=& \frac{4m-2-2\eps}{1-m} \cdot \rho^2(|x|) \rho'^2(|x|) \cdot \big\{ \rho^2(|x|) + at \big\}^{-2+\eps}
	+ \rho(|x|) \rho''(|x|) \cdot \big\{ \rho^2(|x|) + at \big\}^{-1+\eps} \\
	& & + \Big\{ \rho'^2(|x|) + \frac{n-1}{|x|} \rho(|x|) \rho'(|x|) \Big\} \cdot
		\big\{ \rho^2(|x|) + at \big\}^{-1+\eps},
	\qquad x\in\Om, \ t>0.
  \eas
  As $\eps<1$ and $m<1$, we can here estimate
  \bea{55.7}
	\hs{-10mm}
	I(x,t)
	&\le& \Big( \frac{4|m|}{1-m}+1\Big) \rho^2(|x|) \rho'^2(|x|) \cdot \rho^{2(-2+\eps)}(|x|)
	+ \rho(|x|) \cdot |\rho''(|x|)| \cdot \rho^{2(-1+\eps)}(|x|) \nn\\
	& & + \Big\{ \rho'^2(|x|) + \frac{n-1}{|x|} \rho(|x|) \rho'(|x|)\Big\} \cdot \rho^{2(-1+\eps)}(|x|) \nn\\
	&=& \Big(\frac{4|m|}{1-m} +1\Big) \rho^{-2+2\eps}(|x|) \rho'^2(|x|)
	+ \rho^{-1+2\eps}(|x|) |\rho''(|x|)|
	+ \frac{n-1}{|x|} \rho^{-1+2\eps}(|x|) \rho'(|x|)
  \eea
  for all $x\in\Om$ and $t>0$,
  whence using that
  \bas
	\rho'(r)=\frac{1}{\eps} r^\frac{1-\eps}{\eps}
	\quad \mbox{and} \quad
	\rho''(r)=\frac{1-\eps}{\eps^2} r^\frac{1-2\eps}{\eps}
	\qquad \mbox{for all } r\in (0,r_0)
  \eas
  by (\ref{55.2}), we find that near the spatial origin,
  \bea{55.8}
	I(x,t)
	&\le& \Big(\frac{4|m|}{1-m}+1\Big) \cdot |x|^\frac{-2+2\eps}{\eps} \cdot \frac{1}{\eps^2} |x|^\frac{2(1-\eps)}{\eps}
	+ |x|^\frac{-1+2\eps}{\eps} \cdot \frac{1-\eps}{\eps^2} |x|^\frac{1-2\eps}{\eps}
	+ \frac{n-1}{|x|} \cdot |x|^\frac{-1+2\eps}{\eps} \cdot\frac{1}{\eps} |x|^\frac{1-\eps}{\eps} \nn\\
	&=& \Big(\frac{4|m|}{1-m}+1\Big) \cdot\frac{1}{\eps^2} + \frac{1-\eps}{\eps^2} + \frac{n-1}{\eps} \nn\\
	&\le& a
	\qquad \mbox{for all $x\in B_{r_0}(0)$ and } t>0
  \eea
  according to (\ref{55.3}).
  In the corresponding complementary region, recalling our definitions of $c_1, c_2, c_3$ and $c_4$ we obtain from (\ref{55.7}) that
  since $\rho$ is nondecreasing and hence $c_1\le\rho\le c_2$ on $[r_0,\infty)$, in line with (\ref{55.4}) we have
  \bas
	I(x,t)
	&\le& \Big(\frac{4|m|}{1-m}+1\Big) c_1^{-2+2\eps} c_3^2
	+ \max\big\{ c_1^{-1+2\eps} \, , \, c_2^{-1+2\eps}\big\} \cdot c_4
	+ \frac{n-1}{r_0} \cdot  \max\big\{ c_1^{-1+2\eps} \, , \, c_2^{-1+2\eps}\big\} \cdot c_3 \\
	&\le& a
	\qquad \mbox{for all $x\in\Om\sm B_{r_0}(0)$ and } t>0.
  \eas
  Together with (\ref{55.7}) and (\ref{55.6}), again since $\eps<1$ and $m<1$ this shows that $\ou_t-\na\cdot (\ou^{m-1}\na\ou)\ge 0$
  in $\Om\times (0,\infty)$, whence using that $\ou$ is continuous in $\bom\times [0,\infty)$, we may rely on the comparison principle
  from Lemma \ref{lem61} to infer that the locally smooth $L^1$-continuous solution $u$ of (\ref{0}) emanating from
  \bas
	u_0(x):=\ou(x,0)=\rho^\frac{2(1-\eps)}{1-m}(|x|),
	\qquad x\in\Om,
  \eas
  satisfies $u\le \ou$ in $\Om\times (0,\infty)$. Since $\rho(0)=0$ by (\ref{55.2}), and since thus
  \bas
	\ou(0,t)=(at)^\frac{1-\eps}{1-m}
	\qquad \mbox{for all } t>0,
  \eas
  this yields the claim.
\qed
\mysection{Appendix}
\subsection{A refined Poincar\'e inequality. Proof of Lemma \ref{lem11}}
For what follows, given a bounded Lipschitz domain $\Om\subset\R^n$ we recall from \cite[p.~74]{CM} the notion of
{\em capacity $\CC_2(E,G)$ of the condenser $(E;G)$ relative to $\Om$}, as specified in (\ref{52.2}).
Here and below, following the notation addressed in that reference we say that a statement holds {\em 2-quasi-everywhere}
($2-$q.~e.) on a masurable set $E\subset\Om$ if it is valid at each point in $E\sm E_0$ with some suitable set $E_0\subset E$
fulfilling $\CC_2(E_0)=0$, where the $2-$capacity of a set $F\subset\Om$ is defined by
\bas
	\CC_2(F):=\inf \bigg\{
	\io |\na\vp|^2 \ \bigg| \
	\vp\in W_0^{1,2}(\Om), \ \vp\ge 1 \mbox{ in some neighborhood of } F \bigg\}.
\eas
The following property of these quantities has been recorded in \cite[pp.~74-75]{CM}.
\begin{lem}\label{lem52}
  Let $n\ge 1$ and $\Om\subset\R^n$ be a bounded domain with Lipschitz boundary.
  Then there exists $\kappa_1=\kappa_1(\Om)>0$ such that for each $\eta\in (0,\frac{|\Om|}{2})$ we have
  \be{52.1}
	\CC_2(E,G) \ge \frac{\kappa_1}{\zeta_1(\eta)}
	\qquad \mbox{whenever $G\subset\Om$ and $E\subset G$ are measurable with $|E|\ge \del$ and $|G|\le\frac{|\Om|}{2}$},
  \ee
  where $\zeta_1$ is as in (\ref{zeta}), and where for measurable $G\subset\Om$ and $E\subset G$,
  \be{52.2}
	\CC_2(E,G):=\inf \bigg\{ \io |\na\vp|^2 \ \bigg| \
		\vp\in W^{1,2}(\Om), \
		\vp\ge 1 \mbox{ $2$-q.e.~on } E
		\mbox{ and }
		\vp\le 0 \mbox{ $2$-q.e.~on } \Om\sm G
	\bigg\}.
  \ee
\end{lem}
We can thereby develop a classical Poincar\'e inequality into the claimed refined version:\abs
\proofc of Lemma \ref{lem11}. \quad
  According to a classical Poincar\'e inequality, there exists $c_1=c_1(\Om)>0$ such that
  \be{11.3}
	\io |\phi-\ov{\phi}|^2 \le c_1 \io |\na\phi|^2
	\qquad \mbox{for all } \phi\in W^{1,2}(\Om),
  \ee
  where $\ov{\phi}:=\frac{1}{|\Om|} \io \phi$ for $\phi\in L^1(\Om)$.\\
  We now let $\eta>0$ and $\vp\in W^{1,2}(\Om)$ be given such that (\ref{11.1}) holds, and according to \cite[p.~74]{CM}
  we may assume without loss of generality that $\vp$ coincides with its so-called $2$-quasicontinuous representative, that is,
  that for every $\eps>0$ there exists a set $F\subset\Om$ with $\CC_2(F)<\eps$ such that $\vp|_{\Om\sm F}$ is continuous.
  We then let $d\ge 0$ be such that
  \be{11.4}
	\big| \{\vp<d\}\big| \le \frac{|\Om|}{2}
	\qquad \mbox{and} \qquad
	\big| \{\vp>d\}\big| \le \frac{|\Om|}{2},
  \ee
  noting that this can achieved by letting $d:=\sup\big\{ d'>0 \ \big| \ |\{\vp<d'\}| \le \frac{|\Om|}{2} \big\}$, for instance.
  In order to show that then, with $\kappa_1$ and $\zeta_1$ taken from Lemma \ref{lem52} and Lemma \ref{lem11}, we have
  \be{11.5}
	\io |\na\vp|^2
	\ge \frac{\kappa_1 d^2}{\zeta_1(\eta)},
  \ee
  we may assume that $d$ is positive, and can then draw on the Lipschitz continuity of $\R\ni\xi\mapsto \xi_+$ to see that
  \be{11.6}
	\psi:=\Big(1-\frac{\vp}{d}\Big)_+
  \ee
  defines a $2$-quasicontinuous function $\psi\in W^{1,2}(\Om)$ which satisfies
  \bas
	\psi=1
	\qquad \mbox{$2$-q.e.~on } E:=\{\vp=0\}
  \eas
  and, with $G:=\{\vp<d\}$,
  \bas
	\psi=0
	\qquad \mbox{$2$-q.e.~on } \Om\sm G=\{\vp\ge d\}.
  \eas
  Since $|E|\ge \eta$ by (\ref{11.1}) and $|G|\le\frac{|\Om|}{2}$ by (\ref{11.4}), in line with (\ref{52.2}) and (\ref{52.1}) we
  obtain that
  \bas
	\io |\na\psi|^2 \ge \CC_2(E,G) \ge \frac{\kappa_1}{\zeta_1(\eta)},
  \eas
  whence (\ref{11.5}) follows upon observing that, on the other hand,
  \bas
	\io |\na\psi|^2 = \frac{1}{d^2} \io |\na\vp|^2
  \eas
  according to (\ref{11.6}).\abs
  Now to derive (\ref{11.2}) from (\ref{11.5}), it is sufficient to note that in the case when $d\ge \ov{\vp}$ we have
  $|\vp-\ov{\vp}| = \vp-\ov{\vp} \ge d-\ov{\vp}$ on $\{\vp\ge d\}$, so that
  \bas
	\io (\vp-\ov{\vp})^2
	\le \int_{\{\vp\ge d\}} (\vp-\ov{\vp})^2
	\ge \frac{|\Om|}{2} \cdot (d-\ov{\vp})^2.
  \eas
  Along with a similar argument addressing the situation when $d<\ov{\vp}$, this shows that regardless of the sign of $d-\ov{\vp}$,
  \bas
	\io (d-\ov{\vp})^2
	\le 2 \io (\vp-\ov{\vp})^2.
  \eas
  Therefore, namely, (\ref{11.5}) together with (\ref{11.3}) implies that
  \bas
	\io \vp^2
	&=& \io \big\{ (\vp-\ov{\vp}) + (\ov{\vp}-d) + d \big\}^2 \\
	&\le& 3 \io (\vp-\ov{\vp})^2
	+ 3 \io (\ov{\vp}-d)^2
	+ 3|\Om| d^2 \\
	&\le& 9 \io (\vp-\ov{\vp})^2
	+ 3|\Om| d^2 \\
	&\le& 9c_1 \io |\na\vp|^2
	+ \frac{3|\Om|}{\kappa_1} \cdot \zeta_1(\eta) \cdot \io |\na\vp|^2,
  \eas
  from which (\ref{11.2}) results with $\kappa(\Om):=\frac{9c_1}{\zeta_1(|Om|)} + \frac{3|\Om|}{\kappa_1}$.
\qed
\subsection{Two auxiliary results based on ODE comparison}
For completeness, let us finally describe how the parameter dependencies of the ODE estimates undertaken in
Lemmata \ref{lem12}, \ref{lem1}, \ref{lem2} and \ref{lem13} can be verified.\abs
We first concentrate on the scenario of a superlinearly damped ODI with constant coefficients, as having appeard in
Lemma \ref{lem12}.
\begin{lem}\label{lem75}
  Let $a>0, b>0, \al>1, t_0\in\R$ and $T>t_0$, and let $y\in C^0([t_0,T)) \cap C^1((t_0,T))$ be nonnegative and such that
  \bas
	y'(t) + ay^\al(t) \le b
	\qquad \mbox{for all } t\in (t_0,T).
  \eas
  Then
  \bas
	y(t) \le \Big(\frac{1}{(\al-1)a}\Big)^\frac{1}{\al-1} \cdot (t-t_0)^{-\frac{1}{\al-1}}
	+ \Big(\frac{b}{a}\Big)^\frac{1}{\al}
	\qquad \mbox{for all } t\in (t_0,T).
  \eas
\end{lem}
\proof
  Since with $A:=(\frac{1}{(\al-1)a})^\frac{1}{\al-1}$ and $B:=(\frac{b}{a})^\frac{1}{\al}$, for
  $\oy(t):=A(t-t_0)^{-\frac{1}{\al-1}} + B$, $t>t_0$, we have $\oy(t)\nearrow +\infty$ as $t\searrow t_0$ as well as
  $\oy^\al\ge A^\al (t-t_0)^{-\frac{\al}{\al-1}} + B^\al$, and since thus
  \bas
	\oy'(t) + a\oy^\al(t) - b
	\ge - \frac{A}{\al-1} (t-t_0)^{-\frac{1}{\al-1}-1} + a A^\al (t-t_0)^{-\frac{\al}{\al-1}} + aB^\al -b
	= 0
	\qquad \mbox{for all } t>t_0,
  \eas
  this follows by an ODE comparison argument.
\qed
By quite a similar argument, also the counterpart involving time-dependent coefficients of the form arising
in Lemmata \ref{lem1}, \ref{lem2} and \ref{lem13} can be addressed:
\begin{lem}\label{lem76}
  Let $a>0, b>0, \al>1, \lam\ge 0$ and $\mu\ge 0$ be such that $\mu\le\frac{\al+\lam}{\al-1}$, and suppose that $t_0\in\R, T>t_0$
  and $y\in C^0([t_0,T)) \cap C^1((t_0,T))$ be nonnegative and such that
  \be{76.1}
	y'(t) + a (t-t_0)^\lam y^\al(t) \le b (t-t_0)^{-\mu}
	\qquad \mbox{for all } t\in (t_0,T).
  \ee
  Then
  \be{76.2}
	y(t) \le A \cdot (t-t_0)^{-\frac{\lam+1}{\al-1}}
	\qquad \mbox{for all } t\in (t_0,T),
  \ee
  where
  \be{76.3}
	A:=\max \Bigg\{ \bigg( \frac{2(\lam+1)}{(\al-1)a} \bigg)^\frac{1}{\al-1} \, , \,
		\bigg( \frac{2b}{a} \cdot (T-t_0)^{\frac{\al+\lam}{\al-1} - \mu} \bigg)^\frac{1}{\al} \Bigg\}.
  \ee
\end{lem}
\proof
  Without loss of generality assuming that $t_0=0$, we only need to observe that $\oy(t):=At^{-\frac{\lam+1}{\al-1}}$,
  $t\in (0,T)$, satisfies
  \bas
	\oy'(t) + at^\lam \oy^\al - bt^{-\mu}
	&=& - \frac{\lam+1}{\al-1} A t^{-\frac{\al+\lam}{\al-1}}
	+ a A^\al t^{-\frac{\al+\lam}{\al-1}} - bt^{-\mu} \\
	&=& a A^\al t^{-\frac{\al+\lam}{\al-1}} \cdot \bigg\{
		\Big( \frac{1}{2} - \frac{\lam+1}{(\al-1)a A^{\al-1}}\Big)
	+ \Big(\frac{1}{2}-\frac{b}{a A^\al} t^{\frac{\al+\lam}{\al-1}-\mu} \Big) \bigg\} \\[1mm]
	&\ge& 0
	\qquad \mbox{for all } t\in (0,T)
  \eas
  due to (\ref{76.3}), because $\frac{\al+\lam}{\al-1}-\mu\ge 0$.
  A comparison argument relying on (\ref{76.1}) and the fact that $\oy(t)\nearrow +\infty$ as $t\searrow 0$, namely,
  thus yields (\ref{76.1}).
\qed

\bigskip

{\bf Acknowledgements.} \
 Youshan Tao was supported by the {\em National Natural Science Foundation of China
   (No. 12571222, 12631017)}.  Michael Winkler
 acknowledges support of the {\em Deutsche Forschungsgemeinschaft} (Project No.~444955436).

\end{document}